\documentclass[11pt]{article}

\usepackage{fullpage,amsmath,amssymb,color,hyperref}
\usepackage{algorithm}
\usepackage{enumerate}
\usepackage{algorithmic}

\usepackage{graphicx}
\usepackage{epstopdf}
\usepackage{subfigure}
\graphicspath{{figures/}}
\usepackage{graphicx}   
\usepackage{multirow}
\usepackage{amsmath}
\usepackage{amsfonts}
\usepackage{latexsym}
\usepackage{verbatim}
\usepackage{amsmath}    
\usepackage{amsmath, amssymb, algorithm, algorithmic}
\usepackage[table]{xcolor}
\usepackage{tcolorbox}
\usepackage{multirow}
\newtheorem{assumption}{Assumption}

\usepackage [pagewise,displaymath, mathlines]{lineno}
\usepackage{amssymb}
\usepackage{epstopdf}
\usepackage{amsmath}
\usepackage{epsf}
\usepackage{graphicx}
\usepackage{mathdots}
\DeclareMathOperator*{\minimize}{minimize~}

\DeclareMathOperator*{\argmin}{arg\,min}

\newcommand{\R}{\mathbb{R}}

\newtheorem{lemma}{Lemma}[section]
\newtheorem{corollary}{Corollary}[section]
\newtheorem{theorem}{Theorem}[section]
\newtheorem{proposition}{Proposition}[section]
\newtheorem{remark}{Remark}[section]

\newcommand{\rank}{{\rm rank}\,}

\newcommand{\beq}{\begin{equation}}
\newcommand{\eeq}{\end{equation}}
\newcommand{\beqa}{\begin{eqnarray}}
\newcommand{\eeqa}{\end{eqnarray}}
\newcommand{\beqas}{\begin{eqnarray*}}
\newcommand{\eeqas}{\end{eqnarray*}}
\newcommand{\bi}{\begin{itemize}}
\newcommand{\ei}{\end{itemize}}
\newcommand{\ba}{\begin{array}}
\newcommand{\ea}{\end{array}}

\newcommand{\nn}{\nonumber}

\def\eqnok#1{(\ref{#1})}

\def\vgap{\vspace*{.1in}}

\def\exp{{\rm exp}}

\newcommand{\bbe}{\Bbb{E}}

\newcommand{\bbr}{\Bbb{R}}
\def\w{\omega}

\def\tr{{\rm tr}}

\def\EE{{\mathbb{E}}}
\def\tr{{\mbox{Tr}}}

\title{Fully Zeroth-Order Bilevel Programming via Gaussian Smoothing}
\date{}
\author{Alireza Aghasi\thanks{Department of electrical engineering and computer science (EECS), Oregon State University; email: alireza.aghasi@oregonstate.edu. }~\textsuperscript{$\ddagger$}
\and Saeed Ghadimi\thanks{Department of management science and engineering; University of Waterloo; email: sghadimi@uwaterloo.ca \\ \indent ~\textsuperscript{$\ddagger$}The authors have equally contributed to this work}~\textsuperscript{$\ddagger$}}

\begin{document}
\maketitle

\begin{abstract}
In this paper, we study and analyze zeroth-order stochastic approximation algorithms for solving bilvel problems, when neither the upper/lower objective values, nor their unbiased gradient estimates are available. In particular, exploiting Stein's identity, we first use Gaussian smoothing to estimate first- and second-order partial derivatives of functions with two independent block of variables. We then use these estimates in the framework of a stochastic approximation algorithm for solving bilevel optimization problems and establish its non-asymptotic convergence analysis. To the best of our knowledge, this is the first time that sample complexity bounds are established for a fully stochastic zeroth-order bilevel optimization algorithm.

\end{abstract}

\section{Introduction}

In this paper, we study the stochastic bilevel programming (BLP) problem given by
\begin{align} \label{main_prob_st}
&\min_{x \in X} \left\{\psi(x):= f(x,y^*(x))=\bbe[F(x,y^*(x),\xi)]\right\} \nn \\
& \text{s.t.} \ \  y^*(x) = \argmin_{y \in \bbr^m} \left\{g(x,y)= \bbe[G(x,y,\zeta)]\right\},
\end{align}
where $f$ and $g$ are continuously differentiable functions in $(x,y) \in X \times \bbr^m$, $X \subseteq \bbr^n$ is a closed convex set, and the expectations are taken with respect to independent random vectors $\xi, \zeta$ whose probability distributions are supported on two subspaces with possibly different dimensions. We are particularly interested in the setting where neither explicit knowledge about $f,g$ are available nor their unbiased stochastic derivatives. In this zeroth-order setting, we assume that only noisy evaluations  of $f$ and $g$ are available upon query to an oracle.

The BLP problem was first introduced by Bracken and McGill in the 1970s~\cite{Bracken1973MathematicalPW} followed by a more general
form of problem involving joint constraints of outer and inner variables. This is a fundamental problem in engineering and economics with direct applications in problems such as decision making \cite{lu2016multilevel}, supply chain \cite{sadigh2012manufacturer, roghanian2007probabilistic}, network design \cite{marcotte1992efficient,leblanc1986bilevel}, transportation and planning \cite{clegg2001bilevel, yin2002multiobjective}, and optimal design \cite{bendsoe1995optimization, herskovits2000contact}. More recently, BLP has found applications in many areas of machine learning and artificial intelligence. 

Zeroth-order methods apply to many optimization problems (including the BLP) where for various reasons such as complexity, lack of access to an accurate model, or computational limitations, there is no or limited access to the objective gradient. In these situations the common practice is approximating the gradient using deterministic or randomized finite difference methods \cite{shi2021numerical, nesterov2017random}. The scope of application for these methods is  vastly diverse  (see the remainder of this section for a comprehensive summary of applications in machine learning, and see   \cite{zhang2021directional, aghasi2012joint, aghasi2013geometric, liu2020primer, hu2023gradient, van2022data} for examples in other areas of engineering).

The techniques and algorithms developed in this paper can be considered in several areas of machine learning and artificial intelligence. It is noteworthy that BLP is a very general problem and even some of the other popular programs such as min-max problems (used in adversarial and robust machine learning) can be cast as a BLP. In particular, the min-max problem $\min_{x\in X}~\max_{y\in Y}~f(x,y)$ can be reformulated in a bilevel form as:
\begin{align*} 
\minimize_{x \in X} ~ f(x,y^*(x)) 
 ~~~~~\text{subject to:} ~~~~ \ \  y^*(x) \in \argmin_{y \in Y} ~-f(x,y).
\end{align*}
This motivates a more comprehensive review of the literature related to BLP and zeroth-order optimization in the area of machine learning. A more detailed literature review highlighting this connection is presented below.

\paragraph{BLP in Machine Learning.} Probably, the most well-known application of BLP in modern machine learning is the meta-learning problem \cite{
franceschi2018bilevel, hospedales2021meta,  finn2017model,nichol2018first, chen2022meta}, where previously learned tasks help expedite and automate the learning of new tasks. A good meta-learning model generalizes well to new tasks that have not been encountered during the meta-training process, and learns the new tasks with many fewer training examples -- what is often referred to as \emph{learning to learn} \cite{hospedales2021meta}. Meta-learning can be presented in various forms, such as metric-based \cite{snell2017prototypical, vinyals2016matching}, model-based \cite{ mishra2018simple, santoro2016meta}, and optimization-based \cite{
franceschi2018bilevel, hospedales2021meta,  finn2017model,nichol2018first, chen2022meta}. 
In the context of supervised learning, 
the optimization-based meta-learning is normally handled through a BLP of the form \cite{hospedales2021meta, franceschi2018bilevel}: \vspace{-.3cm}
\begin{align*} 
    \minimize_{\lambda\in\Lambda} & ~~ \sum_{i=1}^M\mathcal{L}^{meta}_i\left(\theta^{*}_i(\lambda), \lambda, \mathcal{D}^{val}_i \right) \\[-.2cm] \mbox{subject to:}&~~  \left\{\theta_1^*(\lambda),\ldots, \theta_M^*(\lambda) \right\}= \argmin_{\{\theta_1,\ldots,\theta_M\}}~ \sum_{i=1}^M\mathcal{L}^{task}_i\left(\theta_i, \lambda, \mathcal{D}^{tr}_i \right).
\end{align*} 
Using $M$ available sample tasks, the ultimate goal of this program is training the models $\mathcal{M}_{\theta_i,\lambda}(.)$ for each task $i=1,\ldots,M$, 
where $\theta_i$ represents the task-specific parameters, and $\lambda$ represents the hyperparameters shared among all task models. The $i$-th sample task can be modeled by the triplet $(\mathcal{D}_i^{tr}, \mathcal{D}_i^{val}, \mathcal{L}_i)$, where $\mathcal{D}_i^{tr}$ and $\mathcal{D}_i^{val}$ are consonant datasets, with $\mathcal{D}_i^{tr}$ used to learn the task-specific parameters $\theta_i$, and $\mathcal{D}_i^{val}$ to learn the common hyperparameters $\lambda$. By sharing  $\lambda$ across all the task models, one hopes to expedite the training process for new tasks.

Beyond meta-learning, BLP is central to many other learning paradigms. In the hyperparameter optimization problem, the bilevel formulation allows a concurrent training of a machine learning model, as well as an optimal selection of the free parameters in the model \cite{bennett2008bilevel, franceschi2018bilevel, mackay2018selftuning}. Automated machine learning (AutoML), which aims to automate the manual steps of a machine learning process is another application of BLP. The BLP in this case often follows a leader/follower format, where the follower takes the optimal action that maximally benefits the
leader’s objective value  \cite{yao2018taking, hutter2019automated}. In continual learning, where one aims to learn continuously from a stream of tasks and data, sharing parameters between the continuous models is a way of updating them quickly, while not letting them forget the previous data patterns \cite{9533824, borsos2020coresets}. In both single-agent and multi-agent reinforcement learning, actor-critic-type methods use bilevel programming \cite{konda1999actor, yang2019provably, hong2020two}. In fact, as stated in \cite{qiu2021finite},  the actor-critic algorithm can be essentially viewed as a solution scheme for a class of BLPs. In multi-agent environments, a BLP formulation can help with coordinating the actions taken by the agents \cite{zhang2020bi}. 
In the neural architecture search problem (seeking to automate the processes of determining the optimal neural network architecture), differentiable architecture search methods, using a bilevel formulation, have shown superior performance and speed \cite{liu2018darts, lian2019towards, wang2020m}.
Clearly, developing a scalable solution strategy to the BLP problem not only contributes to major areas such as meta-learning and AutoML, it can also find immediate applications in subdivisions of machine learning that are application-specific.

\paragraph{Zeroth-Order Methods in Machine Learning.} In many machine learning problems, the stochastic versions of the gradient often replace its exact version for training and analysis purposes. Because of privacy or model complexity concerns, the access to such information can be limited, for which the use of stochastic/zeroth-order methods becomes prominent. In the realm of adversarial machine learning, zeroth-order optimization has become a systematic method for crafting adversarial examples from the so called \emph{black-box} models. Creating adversarial examples is usually done in a rather unrealistic white-box setting, where the attacker has full access to the target model. However, in a real-world setting and in the context of attacking a deployed machine learning model, the attacker can only communicate with the model by sending input queries, and receiving the output feedback, which define a black-box attack setting. Works such as \cite{chen2017zoo, tu2019autozoom, ilyas2018black} use zeroth-order techniques to construct such attacks. In contrast, works such as \cite{zhang2021robustify} use zeroth-order methods to robustify black-box models against such attacks. This specifically applies to models that share no information with the outside world for privacy reasons, and yet need to become resilient to adversarial attacks. 

Zeroth-order techniques can reveal more information about black-box models. 
For example in \cite{dhurandhar2019model, dhurandhar2022model}, the authors use zeroth-order optimization to make black-box classification models explainable and interpretable. Their method only requires class probabilities for a given input and can be applied to a variety of black-box models such as deep neural networks, random forests, and boosted trees. 

Gradient-free optimization, and in particular zeroth-order methods are also considered in meta learning problems which as stated above, often offer an explicit BLP formulation. In \cite{du2019query}, the authors propose a ``meta attack'' approach to reduce the number of queries needed for a black-box attack. In \cite{ruan2019learning}, the authors apply the zeroth-order optimization to the learning to learn (L2L) problem (see \cite{chen2022learning} for a comprehensive review of L2L), where the design of an optimization algorithm is learned through a deep neural network. They use a recurrent neural network as the learned optimizer, which initially provides a gradient approximate through a zeroth-order method, and then uses the previous iterations to produce parameter updates.  Song et al. \cite{song2019maml} tackle the model-agnostic meta-learning problem (e.g., see the seminal paper \cite{finn2017model}) through evolution strategies and zeroth-order techniques that involve finite difference approximations, or Gaussian smoothing \cite{nesterov2017random}. It is also possible to meta-learn zeroth-order optimizers by training a model that rather than learning the optimizer states such as the gradient, learns evaluations of the inner loop in the meta-learning bilevel program \cite{chen2017learning}. 

Zeroth-order methods are applicable to many other areas of machine learning including but not limited to transfer learning  of black-box models \cite{tsai2020transfer}, action space exploration in reinforcement learning \cite{vemula2019contrasting}, neural architecture search \cite{wang2022zarts, xie2023zo}, hyperparameter tuning \cite{li2021zeroth}, and federated learning \cite{fang2022communication, chen2023fine}. The reader is referred to \cite{liu2020primer} for a more comprehensive review of zeroth-order methods in machine learning. 

\paragraph{Related Works and Our Solution Strategy.}Solving a bilevel program can be computationally challenging, especially in large machine learning problems that involve millions or billions of parameters \cite{colson2007overview}. In particular, to compute the gradient of the upper-level objective function, one needs an optimal solution to the lower level objective function which is not readily available. Thus, an approximation of this solution will result in an additional error term. Bounding the accumulation of this error is one of the most challenging parts in designing bilevel optimization algorithms, specifically when obtaining first-order information of the objective functions is itself computationally expensive.

The first general  approach for solving BLPs is to reduce the bilevel problem to a single level problem by replacing the lower
optimization problem with its optimality conditions as constraints (see e.g., \cite{Hansen1992NewBR,Shi2005AnEK}). When the lower problem is large-scale, the reduced
problem will have too many constraints. Moreover, without imposing a special structure on $g$, its optimality conditions usually make the feasible set of the reduced problem nonconvex. In addition, the complementary constraints are combinatorial in nature, which results in a
mixed integer programming problem. Due to these shortcomings, a second approach has been proposed to use iterative algorithms for directly solving a bilevel optimization
problem. Penalty methods solving penalized lower objective function (see e.g., \cite{Aiyoshi1980HIERARCHICALDS,Case1998AnLP}), descent methods using approximate descent feasible directions (see
e.g., \cite{Falk1995OnBP,Kolstad1990DerivativeEA}), and trust region methods with mixed integer from of subproblems (see e.g., \cite{Colson2005ATM}) are among these methods.

The main body of existing works on bilevel optimization have focused on developing (stochastic) methods which require (stochastic) second-order derivatives of the inner function $g$ (e.g., \cite{GhadWang18, chen2021closing,ji2021provably,dagrou2022a,chen2023optimal}). More recently, fully (stochastic) first-order methods have been developed that only require first-order derivatives of both upper and lower objective functions (\cite{Ye2022BOMEBO,kwon2023fully}). However, to the best of our knowledge, there are only few works considering (partially) derivative-free approaches for solving bilevel optimizations. 

In \cite{conn2012bilevel, ehrhardt2021inexact} the authors focus on solving a quadratic belief model, and their approach is mainly applicable to the deterministic setting (non-stochastic and hence not scalable). As a classical approach for zeroth-order optimization, one can approximate the functions by convolution with some probability density functions such as uniform over a ball or Gaussian distributions. It can be shown that the gradient of this approximation function is calculated by the values of the original function. This property is useful to provide a biased gradient estimator for $f$ and $g$ in \eqnok{main_prob_st} by only using their function values. In \cite{Gu-etal2021}, the authors use the Gaussian convolution approach for the upper level function while assuming a good solution to the lower level problem is available through an optimization algorithm such as (proximal) gradient method. However, no convergence analysis is provided. On the other hand, Jacobian/Hessian-free methods have been developed recently for solving bilevel optimization problems which do not require any second-order information about the lower function $g$ while still using the first-order derivatives of both $f$ and $g$ (e.g., \cite{sow2022convergence,yang2023achieving}).

In all of the above-mentioned methods, some first-order information of $f$ or $g$ is used. In this paper, we focus on Gaussian convolution to approximate these functions and provide the first fully zeroth-order method to solve bilevel optimization problems. This setting also occurs in several applications in which function evaluations are only estimated via simulation. Our contribution in this paper consists of the following aspects. First, we generalize Gaussian convolution technique to the functions with two block-variables and establish all corresponding relationships between such functions and their smooth Gaussian approximations. Using two independent smoothing parameters for each block, will give us the flexibility of exploiting zeroth-order derivative estimates over just one block. To the best of our knowledge, this is the first time that such properties have been established for a function of two block-variables. Second, we provide the first fully zeroth-order stochastic approximation method for solving bilevel optimization problems. Unlike the existing works, we do not assume the availability of unbiased first/second order derivatives either for the upper level objective function or the lower level one. 
We also provide a detailed non-asymptotic convergence analysis of our proposed method and present sample complexity results.
To the best of our knowledge, these are the first established sample complexity bounds for a fully zeroth-order method proposed for solving stochastic bilevel optimization problems.

The remainder of the paper is organized as follows. First, we lay out the necessary assumptions and notations in the remaining part of the introduction. Then, to develop the required analysis tools, Section \ref{sec:second} focuses on the application of Gaussian smoothing techniques to functions with two block-variables. In this regard, estimates of the first and second-order partial derivatives are presented and the discrepancies with their true values are carefully bounded in terms of the problem inputs, such as the dimension and smoothing parameters. This part is also independently usable for future developments that apply Gaussian smoothing techniques to functions of two block-variables. In Section \ref{sec:third}, our focus shifts to the zeroth-order approximation of the bilevel optimization in \eqnok{main_prob_st}. We address problems such as the approximation error and the efficient evaluation of the gradient of the upper-level objective function.  Section \ref{sec:four} is devoted to the presentation of our proposed solution algorithm, utilizing tools and results from Sections \ref{sec:second} and \ref{sec:third} to analyze the inner and outer loops of our bilevel programming scheme. We provide sample complexity results pertaining the overall algorithm performance. Finally, some concluding remarks are presented in Section \ref{sec:conc}. To abide by the page limit, the proofs of all the technical results in Sections \ref{sec:second} and \ref{sec:third} are moved to the Appendix. The proofs related to Section \ref{sec:four} are all referenced and presented within the section.

\subsection{Main Assumptions}
As a general assumption and to avoid degenerate cases, throughout the paper we assume $m,n\geq 1$. For the sake of analysis, we need to make some other assumptions about problem \eqref{main_prob_st} that hold throughout the paper.

\begin{assumption}\label{fg_assumption}
The following statements hold for the functions $f$ and $g$.
\begin{itemize}
    \item [a)] The function $f$ is Lipschitz continuous. 
    
     \item [b)] The function $f$ has Lipschitz continuous gradients.
    
     \item [c)] The function $g$ is twice differentiable in $(x,y)$ with Lipschitz continuous gradient and Hessian.
     
     \item [d)] For any fixed $ x \in X$, $g(x,\cdot)$ is strongly convex whose modulus does not depend on $x$.
\end{itemize}
\end{assumption}

We will use parts of the above assumption in different results as will be explained later. Since, we are dealing with the zeroth-order, we need to make the following assumptions about the stochastic estimators of $f$ and $g$.
\begin{assumption}\label{stochastic_assumption}
The stochastic functions $F(x,y,\xi)$ and $G(x,y,\zeta)$ are respectively differentiable and twice differentiable in $(x,y) \in \bbr^{n \times m}$, and the statements of Assumption~\ref{fg_assumption} hold for them as well. Moreover, $\nabla_x F(x,y,\xi)$, $\nabla_y F(x,y,\xi)$, $\nabla_y G(x,y,\zeta)$, $\nabla^2_{xy} G(x,y,\zeta)$, $\nabla^2_{yy} G(x,y,\zeta)$ are unbiased estimators with bounded variance for the true gradients and Hessian of $f$ and $g$.
\end{assumption}
We need to emphasize that the unbiased stochastic gradients of $f$ and $g$ in the above assumption are assumed to be inaccessible throughout the paper and the above assumptions are solely made for analysis purposes. We also need the following assumption to establish the convergence analysis of our proposed method.
\begin{assumption}\label{bnd_yx}
The optimal solution to the lower problem is bounded i.e., $\max_{x \in X} \|y^*(x)\|$ is bounded.
\end{assumption}

\subsection{Notation and Definitions} The vertical concatenation of two vectors $x_1\in\mathbb{R}^{d_1}$ and $x_2\in\mathbb{R}^{d_2}$ (i.e., $[x_1^\top,x_2^\top]^\top$) which is a vector in $\mathbb{R}^{d_1+d_2}$ would be denoted by $(x_1,x_2)$. We reserve $I$ for the identity matrix. 

\paragraph{Smoothness.} Different orders of smoothness are considered in this paper. We use $L_{i,q}$ to show the (expected) Lipschitz constant of $i$-th derivative of a (stochastic) function $q$, where $i=0$ corresponds to the function itself. More specifically,
\begin{itemize}
    \item[--]  $q\in \mathcal{C}^0(S;L_{0,q})$, if for $z,z'\in S$, $|q(z) - q(z')|\leq L_{0,q}\|z-z'\|$.
    \item[--] $q\in \mathcal{C}^1(S;L_{1,q})$, if for $z,z'\in S$, $\|\nabla q(z) - \nabla q(z')\|\leq L_{1,q}\|z-z'\|$, or equivalently,
    \begin{equation*}
        \left |q(z') - q(z) -\left \langle \nabla q(z), z'-z\right \rangle\right|\leq \frac{L_{1,q}}{2}\|z-z'\|^2.
    \end{equation*}
   \item[--] $q\in \mathcal{C}^2(S;L_{2,q})$, if for $z,z'\in S$, $\|\nabla^2 q(z) - \nabla^2 q(z')\|\leq L_{2,q}\|z-z'\|$, or equivalently,
    \begin{equation*}
        \left|q(z') - q(z) -\langle \nabla q(z), z'-z\rangle- \frac{1}{2}\langle \nabla^2 q(z)(z'-z), z'-z\rangle\right |\leq \frac{L_{2,q}}{6}\|z-z'\|^3.
    \end{equation*}
\end{itemize}

\paragraph{Strong Convexity.} We use $\lambda_q$ to show the strong convexity modulus of $q$. More specifically, for a convex domain $S$, when $q\in \mathcal{C}^1(S,\cdot)$: 
\begin{equation*}
   q(z')-q(z)-\langle\nabla q(z),z'-z \rangle \geq \frac{\lambda_q}{2}\|z-z'\|^2, ~~ z,z'\in S,
\end{equation*}
and when  $q\in \mathcal{C}^2(S,\cdot)$: 
\begin{equation*}
    \nabla^2 q(z)\succeq \lambda_q I, ~~ z\in S. 
\end{equation*}
\paragraph{Stochastic Moments.} For a stochastic function $Q(z,\zeta)$, with a random parameter $\zeta$, we use $M_{i,Q}^2$ and $\sigma_{i,Q}^2$ to show global upper bounds on the second moment and variance of its $i$-th derivative, respectively. For example, in this case bounded $M_{0,Q}^2$ and $\sigma_{2,Q}^2$ imply that for all $z$:
\[
\EE_\zeta \left[ Q^2(z,\zeta) \right]\leq M_{0,Q}^2, ~~\mbox{and}~~ \EE_\zeta \left\| \nabla_z^2 Q(z,\zeta) - \EE_\zeta \left[\nabla_z^2 Q(z,\zeta)\right]\right\|_F^2\leq \sigma_{2,Q}^2.
\]
As a generalization of the variance, for a random vector $x$, we define the notation
\[
\mathbb{V}[x] := \EE \|x-\EE[x]\|^2 = \EE\|x\|^2 - \left\|\EE x\right\|^2.
\]

\section{Gaussian Smoothing for Functions of Two Block-Variables }\label{sec:second}
In this section, we collect some facts about the first and second derivative of a given function using only function evaluations.
\subsection{Stein's Identity and Zeroth-Order Smooth Approximation}\label{sec:stein}
To provide a zeroth-order smooth approximation of a function, we first review the Stein's identity theorem \cite{stein1981estimation, stein1972bound}. 
\begin{theorem}\label{thStein}
Let $w\sim\mathcal{N}(0,I_d)$, be a standard Gaussian random vector, and let $q:\mathbb{R}^d\to\mathbb{R}$, be an almost-differentiable function \footnote{See \cite{stein1981estimation} for a definition of almost differentiable}, with $\mathbb{E}[\|\nabla q\|]<\infty$, then:
\begin{equation}\label{stein1}
    \mathbb{E}[w~\!q(w)] = \mathbb{E}[\nabla q(w)].
\end{equation}
Furthermore, when the function $q$ has a twice continuously differentiable Hessian, we have
\begin{equation}\label{stein1}
    \mathbb{E}\left[\left( w w^\top - I_d\right)q(w) \right]  = \mathbb{E}\left[\nabla^2 q(w)\right].
\end{equation}
\end{theorem}
Considering the block structure $w = (u,v)$, where $u\sim\mathcal{N}(0,I_n)$, and $v\sim\mathcal{N}(0,I_m)$, Theorem \ref{thStein} extends to the following results for the partial derivatives of $q(u,v)$:
\begin{equation}
    \mathbb{E}[u~\! q(u,v)] = \mathbb{E}[\nabla_u q(u,v)], ~~~~~ \mathbb{E}[v~\! q(u,v)] = \mathbb{E}[\nabla_v q(u,v)].
\end{equation}
Moreover, focusing on the diagonal blocks of $\nabla^2 q(w)$, one gets
\begin{equation}
    \mathbb{E}\left[\left( u u^\top - I_n\right)q(u,v) \right]  = \mathbb{E}\left[\nabla^2_{uu} \ q(u,v)\right], ~~~~~  \mathbb{E}\left[\left( v v^\top - I_m\right)q(u,v) \right]  = \mathbb{E}\left[\nabla^2_{vv} \ q(u,v)\right],
\end{equation}
and the off-diagonal blocks offer the identity
\begin{equation}
    \mathbb{E}\left[ u v^\top q(u,v) \right]  = \mathbb{E}\left[\nabla^2_{uv} \ q(u,v)\right], ~~~~~  \mathbb{E}\left[  v u^\top q(u,v) \right]  = \mathbb{E}\left[\nabla^2_{vu} \ q(u,v)\right].
\end{equation}

Letting $w\sim\mathcal{N}(0,I_d)$, one way of obtaining a smooth approximation of a given possibly non-smooth function $q(z)$, is through the following Gaussian convolution (e.g., \cite{nesterov2017random}):
\begin{equation}
    q_\nu(z) \triangleq \mathbb{E}_w [q(z+\nu w)],
\end{equation}
where $\nu\in(0,\infty)$ is called the smoothing parameter (the approximation trivially becomes exact when $\nu=0$). Using $\hat q(w) = q(z+\nu w)$ in the Stein's Theorem \ref{thStein}, one can immediately verify that
\begin{align}\notag 
     \nabla q_\nu(z) &= \mathbb{E}_w \left[  \frac{q(z+\nu w)}{\nu} \ w\right]\\ & = \mathbb{E}_w \left[  \frac{q(z+\nu w)-q(z)}{\nu} \ w\right].
\end{align}
Also, for the Hessian of the smooth approximation we get
\begin{align}\label{HessApprox1}
     \nabla^2 q_\nu(z) &= \mathbb{E}_w \left[  \left( ww^\top -I_d\right)\frac{q(z+\nu w)}{\nu^2} \ \right]\\ \label{HessApprox2} & = \mathbb{E}_w \left[  \left( ww^\top -I_d\right) \frac{q(z+\nu w) + q(z-\nu w)-2q(z)}{2\nu^2} \right].
\end{align}
Compared to \eqref{HessApprox1}, the Hessian formulation \eqref{HessApprox2} provides a more controllable way of bounding the second moments when $\nu\to 0$, and would be henceforth used.   

Considering the block structure $z = (x,y)$, one can also consider a smooth approximation to $q(x,y)$ as
\begin{equation}
    q_{\eta,\mu}(x,y) \triangleq \mathbb{E}_{u,v} [q(x+\eta u, y+\mu v)],
\end{equation}
where $u\sim\mathcal{N}(0,I_n)$, $v\sim\mathcal{N}(0,I_m)$, and $\eta\geq 0$ and $\mu\geq 0$ control the level of smoothness along the $x$ and $y$ directions, respectively. The following result states that the main convexity and smoothness characteristics of $q$ are preserved under the proposed smooth approximation. 

\begin{proposition}\label{general_prop}
Consider $x\in\mathbb{R}^n$, $y\in \mathbb{R}^m$, and $z = (x,y)\in \mathbb{R}^{n+m}$. For a given function $q: \mathbb{R}^{n+m}\to \mathbb{R}$, and non-negative scalars $\eta$ and $\mu$:
\begin{itemize}
    \item[(a)] If $q(z)$ is convex (strongly convex with modulus $\lambda_q$), $q_{\eta,\mu}(z)$ is also convex (strongly convex with modulus $\lambda_{q_{\eta,\mu}} = \lambda_q$). 
    \item[(b)]  Assume that $q\in \mathcal{C}^i(\mathbb{R}^{n+m};L_{i,q})$ for a smoothness order $i\in \{0,1,2\}$. Then $q_{\eta,\mu}\in \mathcal{C}^i(\mathbb{R}^{n+m};L_{i,q})$ for the same smoothness order $i$.

\end{itemize}
\end{proposition}
\begin{proof}
    See Section \ref{general_prop:proof} of the Supplement. 
\end{proof}

Using a similar line of argument as above and Section \ref{sec:stein}, the first order partial derivatives of  $q_{\eta,\mu}(x,y)$ can be acquired as
\begin{align}
     \nabla_x q_{\eta,\mu}(x,y)  &= \mathbb{E}_{u,v} \left[  \frac{q(x+\eta u, y+\mu v) - q(x, y)}{\eta} \ u\right],\label{gradx}\\
     \nabla_y q_{\eta,\mu}(x,y) & = \mathbb{E}_{u,v} \left[  \frac{q(x+\eta u, y+\mu v) - q(x, y)}{\mu} \ v\right].\label{grady}
\end{align}
Similarly, for the second order partial derivatives we have
\begin{align}
     \nabla^2_{xx} q_{\eta,\mu}(x,y) &=  \mathbb{E}_{u,v} \left[  \left( uu^\top -I_n\right) \frac{q(x+\eta u, y+\mu v) + q(x-\eta u,y-\mu v)-2q(x,y )}{2\eta^2} \right],\label{Hessian_x}\\
     \nabla^2_{yy} q_{\eta,\mu}(x,y) &=  \mathbb{E}_{u,v} \left[  \left( vv^\top -I_m\right) \frac{q(x+\eta u, y+\mu v) + q(x-\eta u,y-\mu v)-2q(x,y)}{2\mu^2} \right],\label{Hessian_y}
\end{align}
and 
\begin{align}
     \nabla^2_{xy} q_{\eta,\mu}(x,y) &=  \mathbb{E}_{u,v} \left[   uv^\top  \frac{q(x+\eta u, y+\mu v) + q(x-\eta u,y-\mu v) -2q(x,y)}{\eta\mu} \right]\label{Hessian_xy2}.
\end{align}
Depending on the smoothness level of $q$, the approximation errors of $q_{\eta,\mu}$ and $\nabla q_{\eta,\mu}$ can be bounded, as stated next. For these set of results we repeatedly appeal to the following lemma from \cite{nesterov2017random}. 
\begin{lemma}\label{lemmaExpectationNormPower}
Consider $u\sim\mathcal{N}(0,I_n)$. Then
\begin{itemize}
    \item[--] for $p\in[0,2]$, $$\EE\|u\|^p\leq n^{p/2},$$
    \item[--] for $p\geq 2$, 
    \[
    n^{p/2}\leq \EE\|u\|^p\leq (n+p)^{p/2}.
    \]
\end{itemize}
\end{lemma}
We are now ready to state the discrepancy results. 


\begin{proposition}\label{propNestApprox}
Let $x\in\mathbb{R}^n$, $y\in \mathbb{R}^m$, and $z = (x,y)\in \mathbb{R}^{n+m}$. Consider function $q: \mathbb{R}^{n+m}\to \mathbb{R}$, and positive scalars $\eta$ and $\mu$.
\begin{itemize}
    \item[(a)] If $q\in \mathcal{C}^0(\mathbb{R}^{n+m};L_{0,q})$, we have $| q_{\eta,\mu}(x,y) - q(x,y)|\leq L_{0,q}\left(\eta^2n + \mu^2 m\right)^{\frac{1}{2}}$.
    
    \item[(b)] If $q\in \mathcal{C}^1(\mathbb{R}^{n+m};L_{1,q})$, we have $| q_{\eta,\mu}(x,y) - q(x,y)| \leq \frac{L_{1,q}}{2}\left(\eta^2n + \mu^2 m\right)$.
    Moreover, $\|\nabla q_{\eta,\mu}(x,y) - \nabla q(x,y)\| \leq U^{(1)}_x + U^{(1)}_y$, where 
    \begin{align*}
        U^{(1)}_x = \frac{L_{1,q}}{2}\left(\eta(n+3)^{\frac{3}{2}} + \frac{\mu^2}{\eta} m n^{\frac{1}{2}} \right), ~~~~U^{(1)}_y =\frac{L_{1,q}}{2}\left( \frac{\eta^2}{\mu} n m^{\frac{1}{2}}+ \mu(m+3)^{\frac{3}{2}} \right).
    \end{align*}
    \item[(c)] If $q\in \mathcal{C}^2(\mathbb{R}^{n+m};L_{2,q})$, we have $ \|\nabla q_{\eta,\mu}(x,y) - \nabla q(x,y)\| \leq  U^{(2)}_x + U^{(2)}_y$, where 
    \begin{align*}
        U^{(2)}_x = \frac{2L_{2,q}}{3}\left(\eta^2(n+4)^2 + \frac{\mu^3}{\eta} (m+3)^{\frac{3}{2}}n^{\frac{1}{2}}    \right), ~~~~U^{(2)}_y =\frac{2L_{2,q}}{3}\left( \frac{\eta^3}{\mu} (n+3)^{\frac{3}{2}}m^{\frac{1}{2}} +  \mu^2(m+4)^2  \right).
    \end{align*}
    Moreover, $\|  \nabla^2 q_{\eta,\mu}(x,y) - \nabla^2 q(x,y)\| \leq U^{(2)}_{xx} + U^{(2)}_{xy} + U^{(2)}_{yy}$,
    where
    \begin{align*}
    U^{(2)}_{xx} &= \frac{2 L_{2,q}}{3}\left( \eta (n+5)^{5/2} + \eta(n+3)^{3/2} + \frac{\mu^3}{\eta^2}(n+1)(m+3)^{3/2} \right), \\ U^{(2)}_{yy} &= \frac{2 L_{2,q}}{3}\left(\mu (m+5)^{5/2} + \mu(m+3)^{3/2} + \frac{\eta^3}{\mu^2}(m+1)(n+3)^{3/2}\right),\\ U^{(2)}_{xy} &= \frac{2 L_{2,q}}{3}\left(\frac{\eta^2}{\mu}(n+4)^2 m^{1/2} + \frac{\mu^2}{\eta}(m+4)^2 n^{1/2}\right).
    \end{align*}
\end{itemize}
\end{proposition}
\begin{proof}
    See Section \ref{propNestApprox:proof} of the Supplement. 
\end{proof}
Two remarks about this result are in place:
\begin{remark}
    Proposition \ref{propNestApprox} is stated in a very general way, so that bounds on the partial derivatives can also be acquired, by simply picking the upper-bound component with the corresponding subscript. For example, $\|\nabla_x q_{\eta,\mu}(x,y) - \nabla_x q(x,y)\| \leq U^{(1)}_x$, when $q\in \mathcal{C}^1(S;L_{1,q})$. Or, if $q\in \mathcal{C}^2(S;L_{2,q})$, then $\|  \nabla^2_{xy} q_{\eta,\mu}(x,y) - \nabla^2_{xy} q(x,y)\| \leq  U^{(2)}_{xy}$.
\end{remark}

\begin{remark}
    Under stronger assumptions about $q$, the gradient and Hessian bounds in parts (b) and (c) of Proposition \ref{propNestApprox} can be significantly improved. These bounds are stated as supplemental remarks in the proof of Proposition \ref{propNestApprox}. 
\end{remark}

We can also derive similar results under stochastic setting. In particular, assume that
function $q(x,y)$ is defined as 
\[
    q_{\eta,\mu}(x,y) \triangleq \mathbb{E}_{u,v,\zeta} [Q(x+\eta u, y+\mu v,\zeta)],
\]
where $\zeta$ is a random vector independent of Gaussian random vectors $(u,v)$.

\begin{proposition}\label{propNestApprox_stch}
Consider $x\in\mathbb{R}^n$, $y\in \mathbb{R}^m$, and $z = (x,y)\in \mathbb{R}^{n+m}$. For a given function $q: \mathbb{R}^{n+m}\to \mathbb{R}$, and positive scalars $\eta$ and $\mu$:
\begin{itemize}
    \item[(a)]If $q\in \mathcal{C}^1(\mathbb{R}^{n+m};L_{1,q})$, we have
\begin{align*}
\EE \|\nabla_x Q_{\eta,\mu}(x,y,\zeta) - \nabla_x q_{\eta,\mu}(x,y)\|^2 &\le \frac{(L_{1,q}+L_{1,Q})^2}{2}\left(\eta(n+3)^{\frac{3}{2}} + \frac{\mu^2}{\eta} m n^{\frac{1}{2}} \right)^2+ 2\sigma_{1,Q}^2,\\
\EE \|\nabla_y Q_{\eta,\mu}(x,y,\zeta) - \nabla_y q_{\eta,\mu}(x,y)\|^2 &\le \frac{(L_{1,q}+L_{1,Q})^2}{2}\left(\frac{\eta^2}{\mu} n m^{\frac{1}{2}}+ \mu(m+3)^{\frac{3}{2}} \right)^2+ 2\sigma_{1,Q}^2,
\end{align*}
where $\nabla_x Q_{\eta,\mu}(x,y,\zeta)$ and $\nabla_y Q_{\eta,\mu}(x,y,\zeta)$ are defined similar to \eqref{gradx} and \eqref{grady}, respectively, and $q(x,y)$ is replaced by $Q(x, y,\zeta)$.

\item [(b)] If $q\in \mathcal{C}^2(\mathbb{R}^{n+m};L_{2,q})$, we have
\begin{align*}
    \EE \|\nabla_{xy} Q_{\eta,\mu}(x,y,\zeta) - \nabla_{xy} q_{\eta,\mu}(x,y)\|^2 \le& \frac{8 (L_{2,q}+L_{2,Q})^2}{9}\left(\frac{\eta^2}{\mu}(n+4)^2 m^{1/2} + \frac{\mu^2}{\eta}(m+4)^2 n^{1/2}\right)^2\\
    &+2\sigma_{2,Q}^2,
\end{align*}
where $\nabla_{xy} Q_{\eta,\mu}(x,y,\zeta)$ is defined similar to \eqref{Hessian_xy2}, in which $q(x,y)$ is replaced by $Q(x, y,\zeta)$.
\end{itemize}
\end{proposition}
\begin{proof}
    See Section \ref{propNestApprox_stch:proof} of the Supplement. 
\end{proof}
\subsection{Stochastic Gradient of Approximate Smooth Functions}
Equations \eqref{gradx}--\eqref{Hessian_xy2} provide exact expressions for the gradient and Hessian of the smooth function $q_{\eta,\mu}(x,y)$ by only referencing function values. To numerically compute the expectation, one requires infinitely many realizations of the input arguments, which is not practical. Instead, we may consider averaging a smaller number of realizations, as few as a single sample (e.g., when employing stochastic gradient descent). Mathematically speaking, instead of evaluating $\nabla_x q_{\eta,\mu}$ and $\nabla_y q_{\eta,\mu}$ through \eqref{gradx} and \eqref{grady}, the following stochastic versions may be considered: 
\begin{align}\label{stoch:gradx}
     \tilde \nabla_x q_{\eta,\mu}(x,y)  &=  \frac{q(x+\eta u, y+\mu v) - q(x, y)}{\eta} \ u, \\
     \tilde \nabla_y q_{\eta,\mu}(x,y) & =  \frac{q(x+\eta u, y+\mu v) - q(x, y)}{\mu} \ v.\label{stoch:grady}
\end{align}
This section focuses on bounding the error introduced by such approximations in representing the true gradient $\nabla q(x,y)$. More specifically, we are interested in upper-bounding the quantity $\EE \|\tilde \nabla q_{\eta,\mu}(x,y)\|^2$ in terms of $\| \nabla q(x,y)\|^2$ and the problem input parameters, such as the smoothness parameters and the dimensionality. 

Moreover, since calculating the gradient of the upper function in a bilevel problem, as will be specified in the next section, involves access to the Hessian of the lower level function, we will also bound the moments of the Hessian blocks $\EE \|\tilde \nabla_{xx}^2 q_{\eta,\mu}(x,y)\|^2$ where 
\begin{equation}
    \tilde \nabla_{xx}^2 q_{\eta,\mu}(x,y) =  \left( uu^\top -I_n\right) \frac{q(x+\eta u, y+\mu v) + q(x-\eta u,y-\mu v)-2q(x,y )}{2\eta^2}.\label{stochHessxx}
\end{equation}
and $\EE \|\tilde \nabla_{xy}^2 q_{\eta,\mu}(x,y)\|^2$ where 
\begin{equation}
    \tilde \nabla_{xy}^2 q_{\eta,\mu}(x,y) =  uv^\top  \frac{q(x+\eta u, y+\mu v) + q(x-\eta u,y-\mu v)-2q(x,y )}{\eta\mu}.\label{stochHessxy}
\end{equation}
Unlike Proposition \ref{propNestApprox} which is presented generally and for different orders of smoothness, the results in this section are only restricted to the smoothness order that is later considered in the algorithm analysis. Clearly, the techniques employed can be used to generalize the results for other cases.

\begin{proposition}\label{propMomentBounds}
Consider $x\in\mathbb{R}^n$, $y\in \mathbb{R}^m$. For $q: \mathbb{R}^{n+m}\to \mathbb{R}$, where $q\in \mathcal{C}^1(\mathbb{R}^{n+m};L_{1,q})$, consider the stochastic gradient $\tilde \nabla_x q_{\eta,\mu}(x,y)$ in \eqref{stoch:gradx}, where $u$ and $v$ are independent standard normal vectors. Then
\begin{equation*}
\EE  \|\tilde \nabla_x q_{\eta,\mu}(x,y) \|^2  \leq L_{1,q}^2\left(\eta^2(n+6)^3+\frac{\mu^4}{\eta^2}n(m+4)^2 \right) +  4(n+2) \|\nabla_x  q(x, y)\|^2 +  \frac{4\mu^2}{\eta^2}n\|\nabla_y  q(x, y)\|^2.
\end{equation*}
\end{proposition}
\begin{proof}
    See Section \ref{propMomentBounds:proof} of the Supplement. 
\end{proof}
\begin{corollary}\label{corMomentBounds}
    By deriving a similar expression for $\mathbb{E}\|\tilde \nabla_y q_{\eta,\mu}(x,y) \|^2$, as
    \begin{align*}
\mathbb{E}\|\tilde \nabla_y q_{\eta,\mu}(x,y) \|^2 \leq   L_{1,q}^2\left( \mu^2(m+6)^3+\frac{\eta^4}{\mu^2}m(n+4)^2 \right)\! +\!  \frac{4\eta^2}{\mu^2}m \|\nabla_x  q(x, y)\|^2\!+\!4(m+2) \|\nabla_y  q(x, y)\|^2,
\end{align*}
and using the fact that $\|\tilde\nabla q\|^2 = \|\tilde\nabla_x q\|^2+\|\tilde\nabla_y q\|^2$, the full stochastic gradient vector can be bounded as
\begin{align*}
\mathbb{E}\|\tilde \nabla q_{\eta,\mu}(x,y) \|^2 \leq ~ & L_{1,q}^2\left(\eta^2(n+6)^3 +\frac{\mu^4}{\eta^2}n(m+4)^2 + \mu^2(m+6)^3+\frac{\eta^4}{\mu^2}m(n+4)^2 \right)\\ & +  4\left(n+2+\frac{\eta^2}{\mu^2}m\right) \|\nabla_x  q(x, y)\|^2+4\left(m+2+\frac{\mu^2}{\eta^2}n\right) \|\nabla_y  q(x, y)\|^2.
\end{align*}
\end{corollary}
In our next result, we bound the second moments of $\|\tilde \nabla_{xx}^2 q_{\eta,\mu}(x,y) z\|$ for a fixed $z$, and relate it to the problem input parameters and the characteristics of the exact Hessian matrix. 
\begin{proposition}\label{hessMomentBounds}
Consider a given $\theta\in \R^n$, $x\in\mathbb{R}^n$, $y\in \mathbb{R}^m$, and $q: \mathbb{R}^{n+m}\to \mathbb{R}$, where $q\in \mathcal{C}^2(\mathbb{R}^{n+m};L_{2,q})$.
\begin{itemize}
    \item [a)] The stochastic Hessian block \eqref{stochHessxx} obeys
\begin{align*}\notag 
\EE_{u,v|\theta}\Big\|  \tilde \nabla_{xx}^2 q_{\eta,\mu}(x,y) \theta  \Big\|^2  \leq & ~2L_{2,q}^2\left( 2\eta^2(n+16)^4 + \frac{\mu^6}{\eta^4}(m+6)^3 (n+3)\right)\|\theta\|^2\\ \notag &+ \bigg( \frac{15}{2} (n+6)^2\left\|\nabla_{xx}^2 q(x,y)\right\|_F^2 + \frac{3\mu^2}{\eta^2}(3n+13)\|\nabla_{xy}^2q(x,y)\|_F^2 \\& \qquad\qquad\qquad\quad  \qquad\qquad\qquad+ \frac{3\mu^4}{2\eta^4}(m+2)(n+3)\|\nabla_{yy}^2q(x,y)\|_F^2\bigg)\|\theta\|^2.
\end{align*}

    \item [b)] The stochastic Hessian block \eqref{stochHessxy} obeys
\begin{align*}\notag 
\EE_{u,v|\theta}\Big\|  \tilde \nabla_{xy}^2 q_{\eta,\mu}(x,y) \theta  \Big\|^2  \leq & ~8L_{2,q}^2\left[ \frac{\eta^4}{\mu^2}(n+8)^4 + \frac{2\mu^4}{\eta^2} n(m+12)^3\right]\|\theta\|^2\\ \notag &+ \bigg( \frac{6\eta^2}{\mu^2}(n+4)(n+2)\|\nabla^2_{xx} q\|_F^2 + 36(n+2)\|\nabla^2_{xy}q(x,y)\|_F^2 \\& \qquad\qquad\qquad\quad  \qquad\qquad\qquad~~+ \frac{30\mu^2}{\eta^2}n(m+2)\|\nabla^2_{yy}q(x,y)\|_F^2\bigg)\|\theta\|^2.
\end{align*}
\end{itemize}

\end{proposition}
\begin{proof}
    See Section \ref{hessMomentBounds:proof} of the Supplement. 
\end{proof}

\begin{remark}
    Using variable symmetry, the result of Proposition~\ref{hessMomentBounds}.a) can be modified to bound $\EE_{u,v|\theta}\|  \tilde \nabla_{yy}^2 q_{\eta,\mu}(x,y) \theta \|^2$. Such bound, together with part b) of the proposition, would offer a bound for $\EE_{u,v|\theta}\|  \tilde \nabla^2 q_{\eta,\mu}(x,y) \theta \|^2$ (the main Hessian) which is skipped here as it is not used in the remainder of the paper. 
\end{remark}

\section{Zeroth-Order Bilevel Formulation}\label{sec:third}
In this section, we focus on zeroth-order approximation of the bilevel optimization in \eqnok{main_prob_st}. As it is clear from this formulation, the computation of $\psi(x)$ and $\nabla \psi(x)$ both explicitly require access to $y^*(x)$. Specifically, the following statements hold (which follow from the chain rule application to the optimality condition $\nabla_y g(\bar x, \nabla y^*(\bar x)) = 0$; detailed proof is available in \cite{GhadWang18}).     
\begin{lemma}
\label{def_Mxy}Suppose that Assumption~\ref{fg_assumption} holds.
\begin{itemize}
\item [a)] For any $\bar x \in X$, $y^*(\bar x)$ is unique and differentiable and 
\beq\label{grad_ystar}
\nabla y^*(\bar x) = -  \left[\nabla_{yy}^2 g(x,y^*(\bar x))\right]^{-1}\nabla_{xy}^2 g(x,y^*(\bar x))^\top.
\eeq

\item [b)] For any $\bar x \in X$:
\beq\label{grad_f2}
\nabla \psi(\bar x) = \nabla_x f(\bar x, y^*(\bar x))- \nabla_{xy}^2 g(x,y^*(\bar x)) \left[\nabla_{yy}^2 g(x,y^*(\bar x))\right]^{-1} \nabla_y f(\bar x, y^*(\bar x)).
\eeq
\end{itemize}
\end{lemma}
Clearly, $y^*(x)$ is unavailable unless the inner problem has a closed-form solution, which poses computational challenges in applying gradient-type algorithms to the BLP problem. Alternatively, we define an estimator to approximate $\nabla \psi(x)$, which is inspired by the formulation in \eqref{grad_f2}. In particular, for any $x \in X$ and  $y \in \bbr^m$, we define 
\beq
\bar \nabla (f,g)(x,y) := \nabla_x f(x,y)- \nabla_{xy}^2 g(x,y)\left[\nabla_{yy}^2 g(x,y)\right]^{-1}\nabla_y f(x,y).\label{grad_f}
\eeq
Accordingly, the following statements hold (c.f. \cite{GhadWang18}).
\begin{lemma}\label{grad_f_error}
Under Assumption~\ref{fg_assumption}, for any $x \in X$ and $y \in \bbr^m$:
\begin{itemize}
\item [a)] The error of the gradient estimator $\bar \nabla f(x, y)$ depends on the quality of the approximate solution to the inner problem, i.e.,
\beq \label{def_grad_error}
\|\bar \nabla (f,g)(x, y) - \nabla \psi(x)\| \le C_1 \|y^*(x)-y\|,
\eeq
where $C_1>0$ is a constant depending on the problem parameters such as Lipschitz constants of functions $f$, $g$, and their gradients.

\item [b)] The optimal solution of the inner problem, $y^*(x)$, is Lipschitz continuous in $x$.

\item [c)] The gradient of $ \psi(x)$ is Lipschitz continuous.
\end{itemize}

\end{lemma}

Using the notations in the previous section, we are now ready to present a variant of the stochastic bilevel problem in \eqnok{main_prob_st} which can be solved using zeroth-order information of the problem. Letting $\mathbf{\eta}=(\eta_1,\eta_2)$ and $\mathbf{\mu}=(\nu_1,\nu_2)$ be two vectors in $\bbr_+^2$,  and introducing the standard normal vectors $u \sim\mathcal{N}(0,I_n)$, and $v\sim\mathcal{N}(0,I_m)$, a Gaussian smooth approximation to \eqnok{main_prob_st} can be formulated as
\begin{align} \label{main_prob_zst}
&\min_{x \in X} \Big\{\psi_{\eta,\mu}(x):= f_{\eta_1,\mu_1}(x,y_{\eta_2,\mu_2}^*(x))=\bbe[F(x+\eta_1 u,y_{\eta_2,\mu_2}^*(x)+\mu_1 v,\xi)]\Big\} \nn \\
& \text{s.t.} \ \  y_{\eta_2,\mu_2}^*(x) = \argmin_{y \in \bbr^m} \Big\{g_{\eta_2,\mu_2}(x,y)= \bbe[G(x+\eta_2 u,y+\mu_2 v,\zeta)]\Big\},
\end{align}
where the expectations are taken w.r.t $(u,v)$ in addition to $(\xi,\zeta)$.

Having defined the bilevel smooth approximation problem in \eqref{main_prob_zst}, similar to Lemma~\ref{def_Mxy}.b), the gradient of the main objective can be formulated as
\begin{align}\label{def_smooth_hypergrad}
\nabla \psi_{\eta,\mu}(\bar x) =&~ \nabla_x f_{\eta_1,\mu_1}(\bar x, y^*_{\eta_2,\mu_2}(\bar x))\nonumber \\
& - \nabla_{xy}^2 g_{\eta_2,\mu_2}(\bar x, y^*_{\eta_2,\mu_2}(\bar x))\left[\nabla_{yy}^2 g_{\eta_2,\mu_2}(\bar x, y^*_{\eta_2,\mu_2}(\bar x))\right]^{-1}  \nabla_y f_{\eta_1,\mu_1}(\bar x, y^*_{\eta_2,\mu_2}(\bar x)).
\end{align}
In the following result, we establish the relationship between the above problem and the BLP problem in \eqnok{main_prob_st}.

\begin{proposition}\label{zeroth-order bilevel}
Consider the bilevel optimization problem \eqnok{main_prob_st} and its Gaussian smooth approximation \eqref{main_prob_zst}. If $f\in\mathcal{C}^1(X\times \R^m,L_{1,f})$ and $g\in\mathcal{C}^1(X\times \R^m,L_{1,g})$, then
\begin{align}
   & \|y_{\eta_2,\mu_2}^*(x)-y^*(x)\|^2 \leq \frac{2 L_{1,g}}{\lambda_g }\left(\eta_2^2n + \mu_2^2 m\right), \label{distysL1}\\ 
    &\|\nabla \psi_{\eta,\mu}(x) - \nabla \psi(x)\| \le L_{1, f} \sqrt{\frac{2 L_{1,g}}{\lambda_g }\left(\eta_2^2n + \mu_2^2 m\right)}\nonumber \\
    & \qquad \qquad \qquad \qquad \qquad + \frac{L_{1,f}}{2}\left(\eta_1(n+3)^{\frac{3}{2}} + \frac{\mu_1^2}{\eta_1} m n^{\frac{1}{2}} + \frac{\eta_1^2}{\mu_1} n m^{\frac{1}{2}}+ \mu_1(m+3)^{\frac{3}{2}} \right):= \sqrt{A}.\label{grad_diff}
\end{align}
\end{proposition}
\begin{proof}
    See Section \ref{zeroth-order bilevel:proof} of the Supplement. 
\end{proof}
\begin{remark}
    Proposition \ref{zeroth-order bilevel} is part of a more comprehensive result that also covers the case
    $f,g\in \mathcal{C}^0$. The full result is presented and proved in the corresponding proof section. 
\end{remark}

\subsection{Zeroth-Order Hessian Inverse Operation}
An important step in designing most of algorithms for solving BLP problems is to estimate the gradient of the upper level objective function which involves a Hessian matrix inverse calculation in \eqnok{grad_f2} which is computationally expensive. Using Neumann series (e.g., \cite{GhadWang18, Ji2020BilevelOC}) and stochastic gradient descent (SGD) method (e.g., \cite{dagrou2022a, chen2023optimal}) are the two main approaches that have been proposed to overcome this computational challenge. In this subsection, we provide a zeroth-order variant of the SGD algorithm to approximate the Hessian inverse product in \eqref{def_smooth_hypergrad} which is central to the calculation of $\nabla \psi_{\eta,\mu}(\bar x)$. The subroutine to be designed is a part of the overall bilevel algorithm and the desired task to be performed by the subroutine ultimately reduces to calculating 
\begin{equation}\label{zbar:def}
\bar z := \left[\nabla_{yy}^2 g_{\eta_2,\mu_2}(\bar x, \bar y)\right]^{-1}  \nabla_y f_{\eta_1,\mu_1}(\bar x, \bar y),
\end{equation}
for given $\bar x$ and $\bar y$. We henceforth focus on the calculation of \eqref{zbar:def}, while avoiding an actual inversion of the Hessian matrix $\nabla_{yy}^2 g_{\eta_2,\mu_2}(\bar x, \bar y)$. To this end, it is straightforward to see that $\bar z$ is the solution to the following strongly convex program:
\begin{equation} \label{zbar:program}
\bar z = \argmin_z ~J(z) := \frac{1}{2} z^\top \nabla_{yy}^2 g_{\eta_2,\mu_2}(\bar x, \bar y) z -  \nabla_y f_{\eta_1,\mu_1}(\bar x, \bar y)^\top z. 
\end{equation}
To solve \eqref{zbar:program}, we consider a stochastic gradient descent (SGD) approach with the gradient approximation
\begin{align}\notag 
\tilde \nabla J(z) =  \Big( v v^\top - &I_m\Big) \left[\frac{G(\bar x+\eta_2 u, \bar y+\mu_2 v, \zeta) + G(\bar x-\eta_2 u,\bar y-\mu_2 v, \zeta)-2G(\bar x,\bar y,\zeta)}{2\mu_2^2} \right]z\\&   -v'\!\left[  \frac{F(\bar x+\eta_1 u', \bar y+\mu_1 v', \xi) - F(\bar x, \bar y,\xi)}{\mu_1} \ \right],\label{gradJ:def}
\end{align}
where $v,v',u$ and $u'$ are standard Gaussian vectors, and $\zeta$ and $\xi$ are random scalars. The zeroth-order SGD algorithm for solving this specific problem is given below.
\begin{algorithm} [H]
	\caption{The Stochastic Zeroth-order Hessian Inverse Approximation (SZHIA) }
	\label{Hinv:alg1}
	\begin{algorithmic}

\STATE \textbf{Input}:
$\bar x \in X$, $\bar y \in \bbr^m$, smoothing parameters $\eta_1,\eta_2,\mu_1,\mu_2$, and maximum iterations $T$.
\STATE Initialize $z_0$
\STATE {\bf For $\tau=0,\ldots$, $T-1$:}

{\addtolength{\leftskip}{0.2in}

\STATE {Draw i.i.d random vectors $\zeta_{\tau}, \xi_\tau$ and $u_{\tau},v_\tau,u'_{\tau},v'_\tau$ independently from a standard Gaussian distribution, and calculate $\tilde \nabla J(z_\tau)$ using \eqref{gradJ:def}.}\\
 Set $$z_{\tau+1} = z_\tau - \gamma \tilde \nabla J(z_\tau).$$
\\
}

{\bf End}
\STATE Output: $H_{\eta_1,\eta_2,\mu_1,\mu_2} = z_T$.
	\end{algorithmic}
\end{algorithm}

In the next result, we provide convergence analysis of the above algorithm. To avoid unnecessary smoothing and derive tighter and simpler bounds, we specifically consider a smoothing setting where $\eta_1=0$.  
\begin{theorem}\label{sgd:hessInv}
Let $z_T$ be the output of Algorithm~\ref{Hinv:alg1} and $\bar z$ as defined in \eqref{zbar:program}. When the smoothing parameters obey $\eta_1=0$ and  $\eta_2/\mu_2=\mathcal{O}(1)$, and  $0<\gamma < \min\left\{ \frac{2}{\lambda_g+L_{1,g}}, \frac{\lambda_g}{2V_H}\right\}$, we have 
\begin{align}
    \EE \left[\| z_T - \bar z\|^2 \right] &\leq (1-\gamma \lambda_g)^T\|z_0 - \bar z\|^2 + \frac{2\gamma \bar V}{\lambda_g},\label{eqSGD:Var}\\ 
    \|\EE [z_T] - \bar z\|^2 &\leq (1-\gamma \lambda_g)^T\|z_0 - \bar z\|^2,\label{eqSGD:bias} 
\end{align}
where for some numerical constants $c$ and $c'$ 
\begin{align}
V_H =&~ c\left[\mu_2^2 L_{2,G}^2m\left( m+n\right)^3 +  m\left(m+n\right)^2\left(\frac{\sigma_{2,G}^2}{m+n} + L_{2,G}^2\right)\right], \label{def_VH}\\ \notag 
\bar V=&~  c \left[\frac{\mu_2^2 L_{0,f}^2L_{2,G}^2}{\lambda_g^2}m\left( m+n\right)^3 + \frac{L_{0,f}^2}{\lambda_g^2} m\left(m+n\right)^2\left(\frac{\sigma_{2,G}^2}{m+n} + L_{2,G}^2\right)\right]\\ &+  c'\left[\mu_1^2 L_{1,F}^2 m^3 + m \left(\sigma_{1,F}^2 + L_{0,f}^2\right)\right].\label{def_Vbar}
\end{align}
\end{theorem}
\begin{proof}
    See Section \ref{sgd:hessInv:proof} of the Supplement. 
\end{proof}

\section{Solution Method}\label{sec:four}
In this section, we present our proposed algorithm to solve the BLP problem in \eqref{main_prob_zst} that only uses zeroth-order information of both upper and lower objective functions. Our method is formally presented as Algorithm~\ref{alg_ZBSA} below. This algorithm consists of two loops. The inner loop is a zeroth-order SGD algorithm that solves the lower problem in BLP up to some accuracy and outputs $\bar y_k$ as an estimate to the optimal solution of the lower problem which will be used in the outer loop to update the decision variable of the upper level problem $x_k$ by a zeroth-order projected SGD step.

\begin{algorithm}
	\caption{The Zeroth-order Double-loop Stochastic Bilevel Approximation (ZDSBA) Method}
	\label{alg_ZBSA}
	\begin{algorithmic}

\STATE Input:
Smoothing parameters $\eta_1,\eta_2,\mu_1,\mu_2$, a maximum number of iteration $N \ge 1$, an initial solution $x_0 \in X$, $y_0 \in \bbr^m$ nonnegative sequences $\{\alpha_k\}_{k \ge 0}$, $\{\beta_t\}_{t \ge 0}$, and integer sequences $\{t_k\}_{k \ge 0}$ and $\{b_k\}_{k \ge 0}$.


{\bf For $k=0,1,\ldots$, $N$:}

\vgap

{\addtolength{\leftskip}{0.2in}

{\bf For $t=0,1,\ldots, t_k-1$:}

}
\vgap
{\color{black}
{\addtolength{\leftskip}{0.4in}

\STATE Compute partial gradient approximations of $G_{0,\mu}$ w.r.t $y$ according to \eqnok{grady} such that
\beq
\tilde \nabla_y G_{0,\mu_2}^t  = \left[ \frac{G(x_k, y_t+\mu_2 v^{(1)}_t, \zeta^{(1)}_t) - G(x_k, y_t,\zeta^{(1)}_t)}{\mu_2}\right] v^{(1)}_t\label{sgrady}
\eeq
where $\zeta^{(1)}_t$ and $v^{(1)}_t$ are i.i.d samples from $\zeta$ and Gaussian distribution, respectively. Set
\beq \label{def_yt}
y_{t+1} = y_t- \beta_t \tilde \nabla_y G_{0,\mu_2}^t.
\eeq

}

{\addtolength{\leftskip}{0.2in}
{\bf End}

\STATE  Set $\bar y_k=y_{t_k}$ and compute the partial gradient approximations according to \eqnok{gradx}, and \eqnok{Hessian_xy2}, such that
\begin{align}
\tilde \nabla_x F_{\eta_1,\mu_1}^k  &= \left[ \frac{F(x_k+\eta_1 u^{(2)}_k, \bar y_k, \xi_k) - F(x_k, \bar y_k,\xi_k)}{\eta_1}\right] u^{(2)}_k,\label{sgradfx}\\
\tilde \nabla^2_{xy} G_{\eta_2,\mu_2}^k &= G_{\eta_2,\mu_2}^k u^{(3)}_k {v^{(3)}_k}^\top  \label{Hessian_xy1},
\end{align}
where
\[
G_{\eta_2,\mu_2}^k =  \frac{G(x_k+\eta_2 u^{(3)}_k, \bar y_k+\mu_2 v^{(3)}_k, \zeta^{(3)}_k) + G(x_k-\eta_2 u^{(3)}_k, y_k-\mu_2 v^{(3)}_k, \zeta^{(3)}_k)- 2 G(x_k,\bar y_k, \zeta^{(3)}_k)}{\eta_2 \mu_2},
\]
and $\zeta^{(i)}_k$, $\xi_k$, and $(u^{(i)}_k,v^{(i)}_k)$ are i.i.d samples from $\zeta$, $\xi$, and Gaussian distribution, respectively.
Call the SZHIA method to compute $\tilde H^k_{\bar \eta,\bar \mu}$ with the inputs $(x_k,\bar y_k, \eta_1, \eta_2, \mu_1, \mu_2, T=b_k)$ and set
\beqa
\tilde \nabla \psi_{\bar \eta,\bar \mu}^k \equiv \tilde \nabla_x f_{\bar \eta,\bar \mu}(x_k,\bar y_k,\w_k) &=& \tilde \nabla_x F_{\eta_1,\mu_1}^k  - \tilde \nabla^2_{xy} G_{\eta_2,\mu_2}^k 
\tilde H^k_{\bar \eta,\bar \mu},\label{grad_f_st}
\eeqa
where $(\bar \eta,\bar \mu) = (\eta_1,\eta_2,\mu_1,\mu_2)$ and $\w_k =(\xi_k,\zeta^{(1)}_{[t_k]}, \zeta^{(2)}_k, \zeta^{(3)}_k, u^{(1)}_k, v^{(1)}_k, u^{(2)}_k, v^{(2)}_k, u^{(3)}_k, v^{(3)}_k)$.
Set
\beq \label{def_xk_st}
x_{k+1} = \arg\min_{x \in X} \left\{\langle \tilde \nabla \psi_{\bar \eta,\bar \mu}^k,x -x_k \rangle + \frac{1}{2 \alpha_k}\|x-x_k\|^2 \right\},
\eeq

}
}

{\bf End}
	\end{algorithmic}
\end{algorithm}

To provide convergence analysis of Algorithm~\ref{alg_ZBSA}, we first need to analyze the inner loop on the lower level problem. More specifically, we first focus on the analysis of the lower optimization problem
\begin{align} \label{inner:optim}
  \minimize_{y \in \bbr^m} g_{\eta_2,\mu_2}(x,y)= \bbe[G(x+\eta_2 u,y+\mu_2 v,\zeta)],
\end{align}
under the assumption that for each $x$, $g(x,y)$ is a strongly convex function of $y$. Intuitively, as detailed in the proof of Theorem \eqref{sgd:general},  developing SGD-type convergence results for \eqref{inner:optim} requires bounding the second moment $\EE  \|\tilde \nabla_y g_{\eta_2,\mu_2}(\tilde x,y_{\eta_2,\mu_2}^*) \|^2$ where $x=\tilde x$ is fixed, and $y_{\eta_2,\mu_2}^*$ is the minimizer of $g_{\eta_2,\mu_2}(\tilde x,y)$. From a technical standpoint, based on Proposition \ref{propMomentBounds} this requires globally bounding $\|\nabla_y g_{\eta_2,\mu_2}(\tilde x,y_{\eta_2,\mu_2}^*) \|^2$ (which is already zero since $y_{\eta_2,\mu_2}^*$ is the minimizer) and ${\eta_2^2}\mu_2^{-2}\|\nabla_x g_{\eta_2,\mu_2}(\tilde x,y_{\eta_2,\mu_2}^*) \|^2$ which is the contribution of smoothing along the $x$ coordinate. If the optimization was performed over a bounded domain, a Lipschitz continuity assumption for $g$ (which also transfers to $g_{\eta_2,\mu_2}$) would have been a straightforward way to bound $\|\nabla_x g_{\eta_2,\mu_2}(\tilde x,y_{\eta_2,\mu_2}^*) \|^2$. However, since the domain of the optimization is unbounded, and Lipschitz continuity and strong convexity are conflicting conditions for a function defined over the entire $\R^m$ (they cannot hold concurrently), we would need to perform the smoothing only along the $y$ direction (i.e., setting $\eta_2=0$), so that the term $\eta_2^2\mu_2^{-2}\|\nabla_x g_{\eta_2,\mu_2}(\tilde x,y_{\eta_2,\mu_2}^*) \|^2$ automatically vanishes to zero. 

For this reason we henceforth focus on the optimization
\begin{equation}\label{eq:SGDMinGenShort}
    y_{0,\mu_2}^* = \argmin_{y\in \R^m}~g_{0,\mu_2}(\tilde x,y) := \EE[G(\tilde x,y+\mu_2 v,\zeta)],
\end{equation}
and the SGD updates
\begin{equation}\label{eq:SGDGenShort}
    y_{t+1} = y_t -\beta_t \tilde\nabla_y G_{0,\mu_2}(\tilde x,y_t,\zeta_t ),
\end{equation}
where
\begin{equation}\label{gradG:stochasticShort}
    \tilde\nabla_y G_{0,\mu_2}(\tilde x,y,\zeta) = v\left[ \frac{G(\tilde x, y+\mu_2 v,\zeta) - G(\tilde x, y,\zeta)}{\mu_2}\right].
\end{equation}
Notice that
\begin{align*}
\EE \left[\tilde\nabla_y G_{0,\mu_2}(\tilde x,y,\zeta)\right] &= \mathbb{E}_{v} \mathbb{E}_\zeta ~\left[v \frac{G(\tilde x, y+\mu_2 v,\zeta) - G(\tilde x, y,\zeta)}{\mu_2} \right] \\& = \mathbb{E}_{v} ~\left[ v\frac{g(\tilde x, y+\mu_2 v) - g(\tilde x, y)}{\mu_2} \right]\\& = \nabla_y g_{0,\mu_2}(\tilde x,y).
\end{align*}
The following result characterizes the convergence behavior of the proposed SGD updates. 
\begin{theorem}\label{sgd:general}
		Consider the minimization \eqref{eq:SGDMinGenShort}, and the updates \eqref{eq:SGDGenShort}. Assume that $g$ and $G$ have Lipschitz continuous gradients, and $g$ is strongly convex. 
If for some $\epsilon >0$, the step-size and the number of iterations is set to 
\begin{equation}\label{eq:SDGStepSize}
   \beta_t = \beta(\epsilon) =  \min\left\{\frac{1}{8(m+4) L_{1,G}}, \epsilon\lambda_g, \frac{1}{\lambda_g}\right\}, \quad 
T = \left\lceil \frac{1}{\beta(\epsilon)\lambda_g}\log\frac{1}{\epsilon}\right\rceil,
\end{equation}
we have
\begin{align}\label{inner_loop1}
\EE \| y_T - y_{0,\mu_2}^*\|^2\leq \epsilon \left[\|y_0-y_{0,\mu_2}^*\|^2 + 8\mu_2^2(L_{1,G}^2+L_{1,g}^2)(m+4)^4 + 16(m+4)\sigma_{1,G}^2 \right]. 
\end{align}
\end{theorem}
\begin{proof}
    To avoid disrupting the flow, the proof is provided at the end of this section (see \S \ref{sgd:general:proof}).
\end{proof}

We can now specialize the above results for the output of the inner loop in Algorithm~\ref{alg_ZBSA}.
\begin{lemma}\label{lemma_sgd}
Let $\{y_t\}_{t=0}^{t_k}$ be the sequence generated at the $k$-th iteration of Algorithm~\ref{alg_ZBSA}. If, for some $\epsilon_k >0$,  
\beq\label{alpha_beta_st}
t_k \ge  \left\lceil \max\left\{\frac{8(m+4) L_{1,G}}{\lambda_g}, \frac{1}{\epsilon_k \lambda_g^2}, 1\right\}\log\frac{1}{\epsilon_k}\right\rceil, \quad \beta_t = \beta(\epsilon_k) \quad \forall t,
\eeq
where $\beta(\cdot)$ is defined in \eqref{eq:SDGStepSize}, we have
\begin{align}\label{y_cnvrg_st}
\EE \left[\|y_{t_k} - y_{\eta_2,\mu_2}^*(x_k) \|^2 \right] &\le 8 \epsilon_k \left[\|y_0-y^*(x_k)\|^2 + 8(m+4)\sigma_{1,G}^2 \right]\nn \\
& \quad + 4\mu_2^2 \left[8\epsilon_k (L_{1,G}^2+L_{1,g}^2)(m+4)^4 + \frac{ L_{1,g}}{\lambda_g } \left((3+4\epsilon_k)m+ \frac{\eta_2^2n}{\mu_2^2}\right)\right]:=\bar A_k. 
\end{align}
\end{lemma}
\begin{proof}
    To avoid disrupting the flow, the proof is provided at the end of this section (see \S \ref{lemma_sgd:proof}).
\end{proof}

We are now ready to present the main convergence results of Algorithm~\ref{alg_ZBSA}.

\begin{theorem}[Convergence results for the ZDSBA algorithm]\label{theom_main_bsa}
Suppose that $\{\bar y_k, x_k \}_{k \ge 0}$ is generated by Algorithm~\ref{alg_ZBSA}, Assumptions~\ref{fg_assumption} and ~\ref{stochastic_assumption} hold, $t_k$ and $\beta_t$ are chosen according to to \eqref{alpha_beta_st}.
\begin{itemize}
\item [a)] Assume that $\psi$ is strongly convex. Then for any $N \ge 1$, we have
\begin{align}
\bbe[\|x^*-x_{k+1}\|^2]
&\le \left(1-\frac{3\alpha_k \lambda_\psi}{4}\right)\bbe[\|x^*-x_{k}\|^2] + \frac{4 \alpha_k}{\lambda_\psi} \left[\tilde A_k+C_1^2 \bar A_k +A \right] \nn \\
&+ \alpha_k^2 \left[(1+2L_{1,\psi}\alpha_k)\sigma_f^2+2L_{1,\psi} C_2 \alpha_k \right],\label{bg_strcvx_st1}\\
\bbe[\psi(\hat x_N)]-\psi^* &\le \frac{1}{\sum_{k=0}^{N-1} \alpha_k^{-1}} \sum_{k=0}^{N-1} \Biggl\{\frac{1}{2 \alpha_k^2}\EE \left[\|x^*-x_k\|^2-\|x^*-x_{k+1}\|^2\right]\nn\\
& \qquad \qquad \qquad \qquad +\left( \frac{2+\lambda_\psi \alpha_k}{4}\right) \left[(1+2L_{1,\psi}\alpha_k)\sigma_f^2+2L_{1,\psi} C_2 \alpha_k \right]\nn \\
&\qquad \qquad \qquad \qquad +\left(\frac{2+\lambda_\psi \alpha_k}{\lambda_\psi \alpha_k}\right) \left[C_1^2 \bar A_k+\tilde A_k +A \right]\Biggl\},\label{bg_strcvx_st2}
\end{align}
where $\bar A_k$, $A$, $\tilde A$, $\sigma_f^2$, are, respectively, defined in \eqnok{y_cnvrg_st}, \eqnok{grad_diff}, \eqnok{def_tildeAk}, \eqnok{sigma_f:eq}, $C_1, C_2$ are positive constants, and
\beq \label{def_xave2}
\hat x_N = \frac{\sum_{k=0}^{N-1} \alpha_k^{-1} x_k}{\sum_{k=0}^{N-1} \alpha_k^{-1}}
\eeq
\item [b)] If $\psi$ is convex, $X$ is bounded with diameter $D_X$, and 
\beq\label{alpha_beta_st01}
\alpha_k \le \frac{1}{4 L_{1,\psi}}  \quad \forall k \ge 0,
\eeq
we have
\begin{align} \label{bg_cvx_st}
\EE[\psi(\hat x_N)]-\psi^* &\le \frac{1}{\sum_{k=0}^{N-1} \alpha_k^{-1}} \sum_{k=0}^{N-1} \Bigg\{\frac{1}{2 \alpha_k^2}\left[\|x^*-x_k\|^2-\|x^*-x_{k+1}\|^2\right.]\nn\\
& \qquad \qquad \qquad \qquad+\frac{D_X}{\alpha_k} \left(\sqrt{\tilde A_k}+C_1 \sqrt{\bar A_k} +\sqrt{A}\right ) +\frac{3\sigma_f^2+C_2}{2}\Bigg\}.
\end{align}

\item [c)] If $\psi$ is possibly nonconvex, $X=\bbr^n$ (for simplicity), and \eqnok{alpha_beta_st01} holds, we have
\begin{align}\label{bg_nocvx}
\bbe[\|\nabla \psi(x_R)\|^2] \le \frac{2[\psi(x_0)-\psi^*]+ \sum_{k=0}^{N-1} \left[3\alpha_k \left(A+C_1^2 \bar A_k+\tilde A_k\right)+
2L_{1,\psi} \alpha_k^2 (\sigma_f^2+C_2) \right]}{\sum_{k=0}^{N-1} \alpha_k} ,
\end{align}
where the expectation is also taken with respect to the integer random variable $R$ whose probability mass function is given by
\beq
Prob(R=k) = \frac{\alpha_k}{\sum_{\tau=0}^{N-1}\alpha_{\tau}} \qquad k=0,1,\ldots, N-1.\label{def_R}
\eeq
\end{itemize}
\end{theorem}

\begin{proof}
We first show part a). Noting that subproblem \eqnok{def_xk_st} is strongly convex, we have
\[
\langle \tilde \nabla \psi_{\bar \eta,\bar \mu}^k ,x_{k+1}-x \rangle \le \frac{1}{2 \alpha_k}\left[\|x-x_k\|^2-\|x-x_{k+1}\|^2-\|x_{k+1}-x_k\|^2 \right] \ \ \forall x \in X.
\]
Moreover, noting the smoothness of $\psi$ due to Lemma~\ref{grad_f_error}.c), we have
\beq\label{psi_smooth1}
\psi(x_{k+1}) \le \psi(x_k) + \langle \nabla \psi(x_k),x_{k+1}-x_k \rangle +\frac{L_{1,\psi}}{2}\|x_{k+1}-x_k\|^2.
\eeq
Adding the above inequalities, denoting $\tilde \Delta_k \equiv  \tilde \nabla \psi_{\bar \eta,\bar \mu}^k-\nabla \psi(x_k)$, and re-arranging the terms,  we obtain
\begin{align}
\psi(x_{k+1}) \le \psi(x_k) &+ \langle \nabla \psi(x_k),x-x_k \rangle +\frac{1}{2 \alpha_k}\left[\|x-x_k\|^2-\|x-x_{k+1}\|^2\right] \nn \\
&- \frac{(1-L_{1,\psi} \alpha_k)}{2 \alpha_k} \|x_{k+1}-x_k\|^2 +  \langle \tilde \Delta_k ,x-x_{k+1} \rangle \quad \forall x \in X. \label{BG_proof1}
\end{align}
By setting $x=x^*$, noting strong convexity of $\psi$, and re-arranging the terms,
we obtain
\begin{align}
\frac{\lambda_\psi}{2}\|x^*-x_{k+1}\|^2 &\le \psi(x_{k+1})-\psi(x^*) \le \frac{1}{2 \alpha_k}\left[\|x^*-x_k\|^2-\|x^*-x_{k+1}\|^2\right] - \frac{\lambda_\psi}{2}\|x^*-x_k\|^2 \nn \\
&\qquad \qquad \qquad \qquad \qquad - \frac{(1-L_{1,\psi} \alpha_k)}{2 \alpha_k} \|x_{k+1}-x_k\|^2
 +  \langle \tilde \Delta_k ,x^*-x_{k+1} \rangle. \label{bsa_proof1}
\end{align}
Now, denoting $\tilde \delta_k =\tilde \nabla \psi_{\bar \eta,\bar \mu}^k - \bar \nabla (f_{\eta_1,  \mu_1}, g_{\eta_2,\mu_2})(x_k,\bar y_k)$,
$\delta_k = \bar \nabla (f_{\eta_1,  \mu_1}, g_{\eta_2,\mu_2})(x_k,\bar y_k)-\nabla \psi_{\bar \eta,\bar \mu} (x_k)$, and $\Delta_k = \nabla \psi_{\bar \eta,\bar \mu} (x_k)-\nabla \psi(x_k)$, 
we have $\tilde \Delta_k = \tilde \delta_k+ \delta_k +\Delta_k$, where $\bar \nabla (f,g)(\cdot, \cdot)$ is defined in \eqref{grad_f}. We also use $\bbe_{|k}$ to denote the conditional expectation given all the generated random sequences up to the $k$-th iteration of Algorithm~\ref{alg_ZBSA}. It is then straightforward to see that 
\begin{align}
\bbe_{|k}[\langle \tilde \Delta_k ,x^*-x_{k+1} \rangle] & = \langle \delta_k +\Delta_k , \bbe_{|k}[x^*-x_{k+1}] \rangle + \langle \bbe_{|k}[\tilde \delta_k] ,x^*-x_{k} \rangle
+\bbe_{|k}[\langle\tilde \delta_k ,x_{k}-x_{k+1} \rangle]\nn
\\
&\le \frac{2}{\lambda_\psi} \left[\|\bbe_{|k}[\tilde \delta_k]\|^2+\|\delta_k\|^2 +\|\Delta_k\|^2 \right] +\frac{\lambda_\psi}{8} \left[2\|\bbe_{|k}[x^*-x_{k+1}]\|^2
+\|x^*-x_{k}\|^2\right]\nn\\
& \ \ +\frac{\alpha_k}{2} \bbe_{|k}[\|\tilde \delta_k\|^2]+\frac{1}{2 \alpha_k} \bbe_{|k}[\|x_{k}-x_{k+1}\|^2],\label{delta_bnd}
\end{align}
where the inequality follows from the fact that for arbitrary vectors $u, ~v$ and scalar $k>0$, $\langle u,v \rangle\leq \|u\|\|v\|\leq \frac{k}{2}\|u\|^2 + \frac{1}{2k}\|v\|^2$. If we multiply both sides of \eqref{bsa_proof1} by $2 \alpha_k$, re-arrange the terms, take a conditional expectation of both sides, and take into account \eqref{delta_bnd}, we have
\begin{align}
\bbe_{|k}[\|x^*-x_{k+1}\|^2]
\le& \left(1-\frac{3\alpha_k \lambda_\psi}{4}\right)[\|x^*-x_{k}\|^2] + \frac{4 \alpha_k}{\lambda_\psi} \left[\|\bbe_{|k}[\tilde \delta_k]\|^2\!+\!\|\delta_k\|^2 \!+\!\|\Delta_k\|^2 \right]
+\alpha_k^2 \bbe_{|k}[\|\tilde \delta_k\|^2]\nn \\ 
&+ L_{1,\psi} \alpha_k \bbe_{|k}[\|x_{k}- x_{k+1}\|^2].
\label{bsa_proof2}
\end{align}
Similarly, combining the above relation with \eqref{bsa_proof1}, and \eqref{delta_bnd}, we have
\begin{align}
 \bbe_{|k}[\psi(x_{k+1})]-\psi(x^*) \le& \frac{1}{2 \alpha_k}\left[\|x^*-x_k\|^2- \bbe_{|k}[\|x^*-x_{k+1}\|^2]\right]+  \frac{\alpha_k(2+\lambda_\psi \alpha_k)\bbe_{|k}[\|\tilde \delta_k\|^2]}{4} \nn\\
 &+\left(\frac{2+ \lambda_\psi \alpha_k}{\lambda_\psi}\right) \left[\|\bbe_{|k}[\tilde \delta_k]\|^2+\|\delta_k\|^2 +\|\Delta_k\|^2 \right]\nn\\
 &+ \frac{L_{1,\psi} (2+\lambda_\psi\alpha_k)\bbe_{|k}[\|x_{k}- x_{k+1}\|^2]}{4} .\label{bsa_proof3}
\end{align}
Moreover, by \eqref{def_xk_st}, non-expansiveness of Euclidean projection, and under Assumption~\ref{fg_assumption}, we have 
\begin{align}
\bbe_{|k}[\|x_{k+1}-x_k\|^2] \le \alpha_k^2 \bbe_{|k}[\|\tilde \nabla \psi_{\bar \eta,\bar \mu}^k\|^2]
&\le 2\alpha_k^2 \left(\bbe_{|k}[\|\tilde \delta_k\|^2]+\| \bar \nabla (f_{\eta_1,  \mu_1}, g_{\eta_2,\mu_2})(x_k,\bar y_k)\|^2 \right)\nn\\
&\le 2\alpha_k^2 \left(\bbe_{|k}[\|\tilde \delta_k\|^2]+ C_2\right),\label{xk-diff}
\end{align}
where $C_2 = 2L_{0,f}^2 \left[1+\frac{ L_{1,g}^2} {\lambda_g^2}\right]$. 

We now bound the terms on the RHS of the above inequalities. In the rest of this proof, for simplicity, we use $h^k \equiv h(x_k,\bar y_k)$ for any function $h$. Noting \eqref{gradx} and \eqref{Hessian_xy2}, we have
\begin{align*}
 \bbe_{|k}[\tilde \delta_k] &=   \bbe_{|k}[\tilde \nabla_x F_{\eta_1,\mu_1}^k - \nabla_x f^k_{\eta_1,\mu_1}] \\
&\quad + \bbe_{|k}\left[\nabla^2_{xy} g^k_{\eta_2,\mu_2}[\nabla^2_{yy} g^k_{\eta_2,\mu_2}]^{-1}\nabla_y f^k_{\eta_1,\mu_1}- \tilde \nabla^2_{xy} G_{\eta_2,\mu_2}^k 
\tilde H^k_{\bar \eta,\bar \mu}\right] \\
&= \nabla^2_{xy} g^k_{\eta_2,\mu_2}\left([\nabla^2_{yy} g^k_{\eta_2,\mu_2}]^{-1}\nabla_y f^k_{\eta_1,\mu_1}- \bbe_{|k}\left[\tilde H^k_{\bar \eta,\bar \mu}\right]\right)
+\bbe_{|k}\left[\left(\nabla^2_{xy} g^k_{\eta_2,\mu_2} - \tilde \nabla^2_{xy} G_{\eta_2,\mu_2}^k\right)
\tilde H^k_{\bar \eta,\bar \mu}\right] \\
&= \nabla^2_{xy} g^k_{\eta_2,\mu_2}\left([\nabla^2_{yy} g^k_{\eta_2,\mu_2}]^{-1}\nabla_y f^k_{\eta_1,\mu_1}- \bbe_{|k}\left[\tilde H^k_{\bar \eta,\bar \mu}\right]\right),
\end{align*}
which in the view of Theorem~\ref{sgd:hessInv}.b) (assuming $z_0=0$ for simplicity), implies that 
\beq
\|\bbe_{|k}[\tilde \delta_k]\| \le (1-\gamma \lambda_g)^\frac{b_k}{2}\|\nabla^2_{xy} g^k_{\eta_2,\mu_2}\| \|[\nabla^2_{yy} g^k_{\eta_2,\mu_2}]^{-1}\nabla_y f^k_{\eta_1,\mu_1}\|
\le
\frac{L_{1,g} L_{0,f}}{\lambda_g} (1-\gamma \lambda_g)^\frac{b_k}{2}:= \sqrt{\tilde A_k}.\label{def_tildeAk} \\
\eeq
Similarly, we have 
\begin{align}
\bbe_{|k}[\|\tilde \delta_k\|^2] &\le 
3\bbe_{|k} \left[\left\|\tilde \nabla_x F_{\eta_1,\mu_1}^k - \nabla_x f^k_{\eta_1,\mu_1}\right\|^2\right]
+3 \left\|\nabla^2_{xy} g^k_{\eta_2,\mu_2} \left[\nabla^2_{yy} g^k_{\eta_2,\mu_2}\right]^{-1}\nabla_y f^k_{\eta_1,\mu_1}\right\|^2 \nn \\
&\quad+3\bbe_{|k}\left[\left\|\tilde \nabla^2_{xy} G_{\eta_2,\mu_2}^k \tilde H^k_{\bar \eta,\bar \mu}\right\|^2\right].
\label{delta_expec}
\end{align}
Moreover, noting Proposition~\ref{propMomentBounds}, Proposition~\ref{hessMomentBounds}.b), Theorem~\ref{sgd:hessInv}.a), the fact that $\|B\|_F \le \sqrt{\min(n,m)} \|B\|$ for $B \in \bbr^{n \times m}$, and under Assumption~\ref{fg_assumption}, we have
\begin{align*}
&\bbe_{|k} \left[\left\|\tilde \nabla_x F_{\eta_1,\mu_1}^k - \nabla_x f^k_{\eta_1,\mu_1}\right\|^2\right] = \bbe_{|k} \left[\left\|\tilde \nabla_x F_{\eta_1,\mu_1}^k \right\|^2\right]- \left\|\nabla_x f^k_{\eta_1,\mu_1}\right\|^2\\
&\qquad \qquad \qquad \qquad   \le L_{1,f}^2\left(\eta_1^2(n+6)^3+\frac{\mu_1^4}{\eta_1^2}n(m+4)^2 \right) +  4(n+2) (\sigma_{1,F}^2+L_{0,f}^2)\left(1 +  \frac{\mu_1^2}{\eta_1^2}\right):=\hat \sigma_f^2,\\
&\bbe_{|k}\left[\left\|\tilde \nabla^2_{xy} G_{\eta_2,\mu_2}^k
\tilde H^k_{\bar \eta,\bar \mu}\right\|^2\right] \le 8\left(\frac{ L_{0,f}^2}{\lambda_g^2} + \frac{\gamma \bar V}{\lambda_g}\right)\Bigg\{4 L_{2,g}^2\left[ \frac{\eta_2^4}{\mu_2^2}(n+8)^4 + \frac{2\mu_2^4}{\eta_2^2} n(m+12)^3\right] \\
& \qquad \qquad \qquad \qquad \qquad  ~+ 3 (n+2) (\sigma_{2,G}^2+(n+m)L_{1,g}^2) \left[\frac{\eta_2^2}{\mu_2^2}(n+4)+6 +\frac{5\mu_2^2}{\eta_2^2}(m+2)\right]\Bigg\}:=\hat \sigma_g^2.
\end{align*}
Combining the above inequalities we obtain
\begin{align}
\bbe_{|k}[\|\tilde \delta_k\|^2] \le 
3 \left[\hat \sigma_f^2 +\hat \sigma_g^2 + \left(\frac{L_{1,g} L_{0,f}}{\lambda_g} \right)^2\right]:=\sigma_f^2.\label{sigma_f:eq}
\end{align}
In the view of Proposition~\ref{zeroth-order bilevel}, Lemma~\ref{grad_f_error}.a), and Lemma~\ref{lemma_sgd}, we also have
\begin{align}
\bbe_{|k}[\|\Delta_k\|^2] \le A,\qquad 
\bbe_{|k}[\|\delta^k\|^2] \le C_1^2 \EE_{|k} \left[\|y_{t_k} - y_{\eta_2,\mu_2}^*(x_k) \|^2 \right] \le C_1^2 \bar A_k.\label{def_A_delta}
\end{align}
Combining the above bounds with \eqref{bsa_proof2}, we obtain \eqref{bg_strcvx_st1}. Similarly, we have
\begin{align*}
 \bbe_{|k}[\psi(x_{k+1})]-\psi(x^*) \le& \frac{1}{2 \alpha_k}\left[\|x^*-x_k\|^2- \bbe_{|k}[\|x^*-x_{k+1}\|^2]\right]+ \left(\frac{2+ \lambda_\psi \alpha_k}{\lambda_\psi}\right) \left[C_1^2 \bar A_k+\tilde A_k +A \right]\nn\\
 &+\alpha_k \left( \frac{2+\lambda_\psi \alpha_k}{4}\right) \left[(1+2L_{1,\psi}\alpha_k)\sigma_f^2+2L_{1,\psi} C_2 \alpha_k \right].
\end{align*}
Dividing both sides by $\alpha_k$, summing them up, and taking a total expectation, we obtain \eqref{bg_strcvx_st2}.

We now show part b). Similar to \eqnok{delta_bnd}, by boundedness of $X$, \eqref{xk-diff}, \eqref{def_tildeAk} , \eqref{sigma_f:eq}, and \eqref{def_A_delta}, we obtain
\begin{align*}
\bbe_{|k}[\langle \tilde \Delta_k ,x^*-x_{k+1} \rangle] & \le
(\|\delta_k +\Delta_k\|) \|\bbe_{|k}[x^*-x_{k+1}]\| + \| \bbe_{|k}[\tilde \delta_k]\|\|x^*-x_{k}\|
+\bbe_{.|k}[\|\tilde \delta_k\| \|x_{k}-x_{k+1}\|]\\
&\le D_X\left(\|\bbe_{|k}[\tilde \delta_k]\|+\|\delta_k\| +\|\Delta_k\| \right) +\alpha_k \bbe_{|k}[\|\tilde \delta_k\|^2]+\frac{1}{4 \alpha_k} \bbe_{|k}[\|x_{k}-x_{k+1}\|^2]\nn\\
&\le  D_X\left(\sqrt{\tilde A_k}+C_1 \sqrt{\bar A_k} +\sqrt{A} \right) +\frac{\alpha_k}{2} \left(3\sigma_f^2+C_2\right),
\end{align*}
which together with \eqref{bsa_proof1} (with $\lambda_\psi =0$ due to the convexity of $\psi$),
imply \eqnok{bg_cvx_st}.

To show part c), assuming that $X=\bbr^n$, noting \eqref{def_xk_st} and \eqref{psi_smooth1}, we obtain 
\begin{align*}
\bbe_{|k}[\psi(x_{k+1})] &\le \psi(x_k) -\alpha_k  \langle \nabla \psi(x_k) , \bbe_{|k} [\tilde \nabla \psi_{\bar \eta,\bar \mu}^k] \rangle
+\frac{L_{1,\psi} \alpha_k^2}{2}\bbe_{|k}[\|\tilde \nabla \psi_{\bar \eta,\bar \mu}^k\|^2]\\
&= \psi(x_k) -\alpha_k \|\nabla \psi(x_k)\|^2 -\alpha_k \langle \nabla \psi(x_k), \bbe_{|k}[\tilde \Delta_k] \rangle
+\frac{L_{1,\psi} \alpha_k^2}{2}\bbe_{|k}[\|\tilde \nabla \psi_{\bar \eta,\bar \mu}^k\|^2]\\
&\le \psi(x_k) -\frac{\alpha_k}{2}\|\nabla \psi(x_k)\|^2 +\frac{\alpha_k}{2} \left\|\bbe_{|k}[\tilde \Delta_k]\right\|^2
+\frac{L_{1,\psi} \alpha_k^2}{2}\bbe_{|k}[\|\tilde \nabla \psi_{\bar \eta,\bar \mu}^k\|^2],
\end{align*}
which similar to the above-mentioned arguments and the fact that
\[
\bbe[\|\nabla \psi(x_R)\|^2] = \sum_{k=0}^{N-1} \frac{\alpha_k}{\sum_{k'=0}^{N-1}\alpha_{k'}}\bbe[\|\nabla \psi(x_k)\|^2],
\]
imply \eqref{bg_nocvx}.
\end{proof}

In the next result, we specialize the rate of convergence of Algorithm~\ref{alg_ZBSA}.

\begin{corollary} \label{lemma_main_bsa}
Suppose that $\{x_k ,\bar y_k, \}_{k \ge 0}$ is generated by Algorithm~\ref{alg_ZBSA}, Assumptions~\ref{fg_assumption} and \ref{stochastic_assumption} hold. Also, assume that $t_k$ and $\beta_k$ are set to \eqnok{alpha_beta_st}, respectively, and smoothing parameters are set to
\begin{align}\label{smooth_params}
    \eta_1=\mu_1 = \frac{1}{\sqrt{(N+1)(n+m)^3}}, \qquad \eta_2=\mu_2 = \frac{1}{\sqrt{(N+1)(n+m)^5}}.
\end{align}
\begin{itemize}
\item [a)] Assume that $\psi$ is strongly convex and
\beq \label{alpha_beta1_st}
\epsilon_k =\frac{(n+m)^2}{k+1}, \ \ b_k = \left\lceil \log_{1-\gamma \lambda_g} \frac{1}{k+1} \right\rceil, \ \ \alpha_k = \frac{4}{\lambda_\psi (k+3)} \ \ \forall k \ge 0.
\eeq
Then, for any $N \ge 1$, we have
\begin{align}
\EE[\|x^*-x_N\|^2] &= {\cal O}\left(\frac{(n+m)^3}{ \lambda_\psi^2 N}\right), \label{bg_strcvx_seq_st2}\\
\bbe[\psi(\hat x_N)]-\psi^* &= {\cal O}\left(\frac{\lambda_\psi \|x^*-x_0\|^2}{N^2}+\frac{(n+m)^3}{\lambda_\psi^2 N}\right).\label{bg_strcvx_st4}
\end{align}

\item [b)] If $\psi$ is convex, $X$ is bounded, an iteration limit $N$ is given, and
\beq \label{alpha_beta1_st2}
\epsilon_k =  \frac{(n+m)^2}{k+1}, \ \ b_k =  \left \lceil \log_{1-\gamma \lambda_g} \frac{1}{k+1} \
\right\rceil , \ \ \alpha_k = \frac{1}{2L_{1,\psi} \sqrt{(n+m)^3(N+1)}},
\eeq
for any given $N \ge 1$, we have
\beq \label{bg_cvx1_st}
\bbe[\psi(\hat x_N)]-\psi^* = 
{\cal O}\left((\|x^*-x_0\|^2 + D_X)\sqrt{\frac{(n+m)^3}{N}}\right).
\eeq

\item [c)] If $\psi$ is possibly nonconvex, $X=\bbr^n$ (for simplicity), $\alpha_k$ is set to \eqnok{alpha_beta1_st2}, and
\beq\label{alpha_beta1_st3}
\epsilon_k =  \sqrt{\frac{n+m}{k+1}}, \qquad  b_k =  \left\lceil \log_{1-\gamma \lambda_g} \frac{1}{\sqrt{k+1}} \right\rceil,
\eeq
we have
\beq\label{bg_nocvx1_st}
\bbe\left[\|\nabla \psi(x_R)\|^2\right] = 
{\cal O}\left((\psi(x_0)-\psi^*)\sqrt{\frac{(n+m)^3}{N}}\right),
\eeq
where the expectation is taken with respect to the integer random variable $R$ uniformly distributed over $\{0,1,\ldots, N-1\}$.
\end{itemize}
\end{corollary}

\begin{proof}
First, note that by \eqref{bg_strcvx_st1} and \eqnok{alpha_beta1_st}, we have
\beq\label{x_ineq}
\bbe \left[\|x^*-x_{k+1}\|^2\right] \le \left(1-\frac{3}{k+3}\right) \bbe \left[\|x^*-x_k\|^2\right] + \frac{16}{\lambda_\psi^2 (k+3)} \left(\tilde A_k+C_1^2 \bar A_k +A  +\frac{C_3}{k+3}\right),
\eeq
where 
\beq\label{C3}
C_3 
 = \left(1+\frac{8L_{1,\psi}}{\lambda_\psi(k+3)}\right)\sigma_f^2 +\frac{8L_{1,\psi} C_2}{\lambda_\psi (k+3)} ={\cal O}\left(\frac{\sigma_f^2}{\lambda_\psi} \right)
\eeq
due to the definition of $C_2$ right after \eqref{xk-diff}.
Now, noting Theorem~\ref{sgd:hessInv}, \eqref{y_cnvrg_st}, \eqref{def_tildeAk}, and \eqref{alpha_beta1_st}, we have 
\begin{align}
\tilde A_k &= \left(\frac{L_{1,g} L_{0,f}}{\lambda_g}\right)^2 (1-\gamma \lambda_g)^{b_k} = {\cal O} \left(\frac{1}{k+1}\right), \qquad \qquad \bar A_k ={\cal O} \left(\frac{m(n+m)^2}{k+1 }\right),\label{A1}\\
A &={\cal O} \left(\frac{1}{N+1}\right),\qquad \qquad
\frac{2\gamma \bar V}{\lambda_g} \le \frac{ \bar V}{V_H} ={\cal O}(1) \to \sigma_f^2 ={\cal O}\left(n (n+m)^2\right).\label{A2}
\end{align}
Multiplying both sides of \eqref{x_ineq} by $(k+1)(k+2)(k+3)$, summing them up and noting the facts that
\[
    \sum_{k=0}^{N-1} (k+1)(k+2)  \left(\tilde A_k+C_1^2 \bar A_k +A  +\frac{C_3}{(k+3)}\right) ={\cal O}\left(N^2 (n+m)^3\right)
    \]
due to \eqref{C3}, \eqref{A1}, and \eqref{A2}, we obtain \eqref{bg_strcvx_seq_st2}
Second, noting the above bound on the generated sequences and \eqnok{alpha_beta1_st}, we have
\begin{align*}
&\sum_{k=0}^{N-1} \frac{1}{2\alpha_k ^2} \left[\|x^*-x_k\|^2-\|x^*-x_{k+1}\|^2\right] \\
&= \frac{\lambda_\psi^2}{4}\|x^*-x_0\|^2+ \sum_{k=1}^{N-1} \left(\frac{1}{2\alpha_k^2} - \frac{1}{2\alpha_{k-1}^2} \right) \|x^*-x_k\|^2- \frac{1}{2\alpha_N^2}\|x^*-x_N\|^2 \\
&=\frac{\lambda_\psi^2}{16} \left[5\|x^*-x_0\|^2+ \sum_{k=1}^{N-1} (2k+5) \|x^*-x_k\|^2 \right]- \frac{1}{2\alpha_N^2}\|x^*-x_N\|^2 \nn \\
&=  {\cal O}\left(\lambda_\psi^2 \|x^*-x_0\|^2+N (n+m)^3\right), \\
&\sum_{k=0}^{N-1} \left(\frac{2+\lambda_\psi \alpha_k}{\lambda_\psi \alpha_k}\right) \left[C_1^2 \bar A_k+\tilde A_k +A\right] = {\cal O}\left(N m(n+m)^2\right), \qquad \sum_{k=0}^{N-1} \alpha_k^{-1} = {\cal O }\left(\lambda_\psi N^2\right),
\end{align*}
which together with \eqref{bg_strcvx_st2}, imply \eqref{bg_strcvx_st4}.
Similarly, we obtain \eqref{bg_cvx1_st} and \eqref{bg_nocvx1_st} for which we skip the details.
\end{proof}

Note that the total noisy evaluations of $f$ and $g$ used in Algorithm~\ref{alg_ZBSA}, are, respectively, bounded by ${\cal O} (N+\sum_{k=0}^{N} b_k)$ and ${\cal O} (\sum_{k=0}^{N} (t_k +b_k))$. Thus, \eqref{alpha_beta_st}, \eqref{alpha_beta1_st}, and \eqref{bg_strcvx_st4} imply that Algorithm~\ref{alg_ZBSA} can find an $\epsilon$-optimal solution of problem~\ref{main_prob_st} within at most ${\cal O}(\frac{(n+m)^3}{\lambda_\psi^2 \epsilon}\log (\frac{n+m}{\lambda_\psi \epsilon}))$ and ${\cal O}(\frac{(n+m)^4}{\lambda_\psi^4 \epsilon^2 }\log (\frac{n+m}{\lambda_\psi \epsilon}))$ number of the noisy evaluations of $f$ and $g$, respectively. Moreover, \eqref{bg_cvx1_st} implies that when the objective function is only convex, the aforementioned complexity bounds increase to ${\cal O}(\frac{(n+m)^3}{\epsilon^2}\log (\frac{n+m}{\epsilon}))$ and ${\cal O}(\frac{(n+m)^4}{\epsilon^4 }\log (\frac{n+m}{\epsilon}))$. Moreover, when $f$ is possibly nonconvex, \eqref{bg_nocvx1_st} implies that Algorithm~\ref{alg_ZBSA} achieves complexity bounds of ${\cal O}(\frac{(n+m)^3}{\epsilon^2}\log (\frac{n+m}{\epsilon}))$ and ${\cal O}\left(\frac{(n+m)^4}{\epsilon^3 }\log \left(\frac{n+m}{\epsilon}\right)\right)$ to achieve an $\epsilon$-stationary point of the problem. To the best of our knowledge, this is the first time that sample complexity bounds have been established for a \emph{fully} zeroth-order stochastic approximation method to address bilevel optimization problems.

It should be also noted that some of the above sample complexity bounds can be improved in terms of the dependency on the target accuracy $\epsilon$, when at least the first-order information of $f$ and $g$ are accessible. The focus of this paper is addressing bilevel problems with the minimum information about the upper/lower objectives, and this alternative problem is left open as a future research direction. Another open direction of research is exploring the  dependency of the sample complexity on the problem dimensions, and the existence of alternative approaches with smaller share of dimensionality.

\subsection{Proof of Theorem \ref{sgd:general}}\label{sgd:general:proof}
The proof is basically a customization of the following lemma, which is proved at the end of this section. 
\begin{lemma}\label{sgd:generalLem}
		Consider the minimization \eqref{eq:SGDMinGenShort}, and the updates \eqref{eq:SGDGenShort}. Assume that $g$ and $G$ have Lipschitz continuous gradients, and $g$ is strongly convex. Assume that for all $y$,
\begin{align*}
		\EE \left\|\tilde\nabla_y G_{0,\mu_2}(\tilde x,y,\zeta)\right \|^2 & \leq \alpha_1 + \alpha_2 \EE \left\|\nabla_y G_{0,\mu_2}(\tilde x,y,\zeta)\right \|^2,
\end{align*} 
  and 
  \begin{align*}
		\EE \left\|\nabla_y G_{0,\mu_2}(\tilde x,y_{0,\mu_2}^*,\zeta)\right\|^2 &\leq \alpha_3,
		\end{align*} 
  for some non-negative constants $\alpha_1, \alpha_2$ and $\alpha_3$. For any $\epsilon>0$, picking 
  \[
   \beta_t = \beta(\epsilon) =  \min\left\{\frac{1}{2\alpha_2 L_{1,G}},\epsilon\lambda_g,\frac{1}{\lambda_g}\right\},~~\mbox{and}~~ T = \left\lceil \frac{1}{\beta(\epsilon) \lambda_g}\log\frac{1}{\epsilon}\right\rceil,  
  \]
  guarantees that $\EE \| y_T - y_{0,\mu_2}^*\|^2\leq \epsilon \left(\|y_0-y_{0,\mu_2}^*\|^2+ \alpha_1 + 2\alpha_2\alpha_3\right)$. 
	\end{lemma}
It only remains to specify the values for $\alpha_1, \alpha_2$ and $\alpha_3$ in Lemma \ref{sgd:generalLem}. Since, the smoothing is only applied along the $y$ direction, we can use some prior results in Gaussian smoothing. More specifically, by Corollary 1 in \cite{nesterov2017random} we have
\begin{align}\label{alpha1alpha2}
		\EE \left\|\tilde\nabla_y G_{0,\mu_2}(\tilde x,y,\zeta)\right \|^2 & \leq \frac{3\mu_2^2 L_{1,G}^2}{2}(m+5)^3 + 4(m+4) \EE \left\|\nabla_y G_{0,\mu_2}(\tilde x,y,\zeta)\right \|^2,
\end{align} 
which characterizes $\alpha_1$ and $\alpha_2$. To characterize $\alpha_3$, note that 
by Proposition \ref{propNestApprox_stch}(a), we have
\begin{align*} 
   \EE \left\|\nabla_y G_{0,\mu_2}(\tilde x,y_{0,\mu_2}^*,\zeta)\right\|^2 &= \EE \left\|\nabla_y G_{0,\mu_2}(\tilde x,y_{0,\mu_2}^*,\zeta) - \nabla_y g_{0,\mu_2}(\tilde x,y_{0,\mu_2}^*)\right\|^2\\
   &\le \frac{\mu_2^2 (L_{1,G}+L_{1,g})^2}{2}(m+3)^3+2\sigma_{1,G}^2. 
\end{align*}
which together with \eqref{alpha1alpha2}, imply that the term $\alpha_1 + 2\alpha_2\alpha_3$ corresponds to
\begin{align}\notag 
    \frac{3\mu_2^2 L_{1,G}^2}{2}(m+5)^3 &+ 8(m+4)\left( \mu_2^2 (L_{1,G}^2+L_{1,g}^2)(m+3)^3 + 2\sigma_{1,G}^2\right)\\&\leq 8\mu_2^2(L_{1,G}^2+L_{1,g}^2)(m+4)^4 + 16(m+4)\sigma_{1,G}^2, \label{rhsSGDStep}
\end{align}
and if we use the RHS of \eqref{rhsSGDStep} as $\alpha_1 + 2\alpha_2\alpha_3$ in Lemma \ref{sgd:generalLem}, we get the proposed step-size in \eqref{eq:SDGStepSize}, which completes the proof. 
\subsubsection{Proof of Lemma \ref{sgd:generalLem}}


Let $\EE_{|t}$ denote the conditional expectation given $y_t$. Subtracting $y_{0,\mu_2}^*$ from both sides of \eqref{eq:SGDGenShort} and taking a squared norm gives
	\begin{align*}
		\EE_{|t}\|y_{t+1}-y_{0,\mu_2}^*\|^2 = \|y_t-y_{0,\mu_2}^*\|^2 -2\beta(y_t-y_{0,\mu_2}^*)^\top  \nabla_y g_{0,\mu_2}(y_k)+ \beta^2 \EE_{|t} \|\tilde\nabla_y G_{0,\mu_2}(\tilde x,y_t,\zeta_t)\|^2. 
	\end{align*}
By Proposition \ref{general_prop}, the strong convexity of $g(\tilde x,y)$ implies the strong convexity of  $g_{0,\mu_2}(\tilde x,y)$ with the same muduli of convexity, i.e.,
\[(y_t-y_{0,\mu_2}^*)^\top  \nabla_t g_{0,\mu_2}(\tilde x,y_t) 	\geq  g_{0,\mu_2}(\tilde x,y_t) - g_{0,\mu_2}(\tilde x, y_{0,\mu_2}^*) + \frac{\lambda_g}{2}\|y_t-y_{0,\mu_2}^*\|^2,
\]
and therefore
\begin{align}\notag 
\EE_{|t}\|y_{t+1}-y_{0,\mu_2}^*\|^2 &\leq (1-\beta \lambda_g)\|y_t-y_{0,\mu_2}^*\|^2 - 2\beta ( g_{0,\mu_2}(\tilde x,y_t) - g_{0,\mu_2}(\tilde x,y_{0,\mu_2}^*) ) \\& \notag ~~~~~~~~~~~~~~~~~~~~~~~~~~~~~~~~\! +  \beta^2 \EE_{|t} \|\tilde\nabla_y G_{0,\mu_2}(\tilde x,y_t,\zeta_t)\|^2\\ \notag &\leq (1-\beta \lambda_g)\|y_t-y_{0,\mu_2}^*\|^2 - 2\beta ( g_{0,\mu_2}(\tilde x,y_t) - g_{0,\mu_2}(\tilde x,y_{0,\mu_2}^*) ) \\ &~~~~~~~~~~~~~~~~~~~~~~~~~~~~~~~~\! +  \beta^2\alpha_1 +  \beta^2\alpha_2\EE_{|t} \|\nabla_y G_{0,\mu_2}(\tilde x,y_t,\zeta_t)\|^2.\label{contractionChain}
\end{align}	
The argument of the expectation in the last term can be bounded as
\begin{align}\notag 
\|\nabla_y  G_{0,\mu_2}(\tilde x,y_t,\zeta_t)\|^2  &\leq 2\left\|\nabla_y G_{0,\mu_2}(\tilde x,y_{0,\mu_2}^*,\zeta_t)\right\|^2+2 \left\|\nabla_y G_{0,\mu_2}(\tilde x,y_t,\zeta_t) - \nabla_y G_{0,\mu_2}(\tilde x,y_{0,\mu_2}^*,\zeta_t)\right\|^2 \\ &\leq 2\left\|\nabla_y G_{0,\mu_2}(\tilde x,y_{0,\mu_2}^*,\zeta_t)\right\|^2 +  4L_{1,G}\left( G_{0,\mu_2}(\tilde x,y_t,\zeta_t)  -  G_{0,\mu_2}(\tilde x,y_{0,\mu_2}^*,\zeta_t)\right) \notag  \\& ~~~~~~~~~~~~~~~~~~~~~~~~~~~~~~~~~~ \! - 4L_{1,G}\left\langle \nabla_y G_{0,\mu_2}(\tilde x, y_{0,\mu_2}^*,\zeta_t),y_t-y_{0,\mu_2}^*\right\rangle,  \label{SGDSecMomExpansion}
\end{align}
where for the second inequality we used the facts that $G$ and $G_{0,\mu_2}$ share identical Lipschitz constants (Proposition \ref{general_prop}), and the well-known implication of smoothness (e.g., see Theorem 2.1.15 of \cite{nesterov2018lectures}) that for a function $q\in \mathcal{C}^1(\R^m;L_{1,q})$,
\[
q(z) + \langle\nabla q(z), z'-z \rangle + \frac{1}{2L_{1,q}}\|\nabla q(z') - \nabla q(z)\|^2 \leq q(z'). 
\]
Applying an expectation to both sides of \eqref{SGDSecMomExpansion} gives 
\begin{align}\notag 
\EE_{|t}\|\nabla_y  G_{0,\mu_2}(\tilde x,y_t,\zeta_k)\|^2 &\leq 2\EE \left\|\nabla_y G_{0,\mu_2}(\tilde x,y_{0,\mu_2}^*,\zeta_t)\right\|^2 +  4L_{1,G}\left( g_{0,\mu_2}(\tilde x,y_t)  -  g_{0,\mu_2}(\tilde x,y_{0,\mu_2}^*)\right) \\& \leq 2\alpha_3 +  4L_{1,G}\left( g_{0,\mu_2}(\tilde x,y_t)  -  g_{0,\mu_2}(\tilde x,y_{0,\mu_2}^*)\right),\label{expectApply2ineq}
\end{align}
where the inner-product term in \eqref{SGDSecMomExpansion} vanishes after taking an expectation since 
\begin{align*}
    \EE_{|t} \nabla_y G_{0,\mu_2}(\tilde x, y_{0,\mu_2}^*,\zeta) &= \EE_\zeta  \nabla_y G_{0,\mu_2}(\tilde x, y_{0,\mu_2}^*,\zeta)\\ &= \EE_\zeta \EE_v \tilde \nabla_y G_{0,\mu_2}(\tilde x, y^*,\zeta)\\ &=\EE_\zeta \EE_v \left[ v\frac{G(\tilde x, y^*+\mu_2 v,\zeta) - G(\tilde x, y_{0,\mu_2}^*,\zeta)}{\mu_2}\right] \\ &= \EE_v \EE_\zeta \left[ v\frac{G(\tilde x, y^*+\mu_2 v,\zeta) - G(\tilde x, y_{0,\mu_2}^*,\zeta)}{\mu_2}\right]\\ & = \EE_v  \left[ v\frac{g(\tilde x, y_{0,\mu_2}^*+\mu_2 v) - g(\tilde x, y_{0,\mu_2}^*)}{\mu_2}\right]\\ &=\nabla_y g_{0,\mu_2}(\tilde x, y_{0,\mu_2}^*)\\&=0.
\end{align*}
Using \eqref{expectApply2ineq} in \eqref{contractionChain} and taking an expectation of both sides with respect to $y_k$ gives
\begin{align}\notag 
\EE\|y_{t+1}-y_{0,\mu_2}^*\|^2 &\leq (1-\beta \lambda_g)\EE \|y_t-y_{0,\mu_2}^*\|^2 - 2\beta(1-2\beta\alpha_2 L_{1,G}) ( g_{0,\mu_2}(\tilde x,y_t) - g_{0,\mu_2}(\tilde x,y_{0,\mu_2}^*) ) \\ \notag &~~~~~~~~~~~~~~~~~~~~~~~~~~~~~~~~~~\! +  \beta^2(\alpha_1 + 2\alpha_2\alpha_3 )\\&\leq (1-\beta \lambda_g)\EE \|y_k-y_{0,\mu_2}^*\|^2 +  \beta^2(\alpha_1 + 2\alpha_2\alpha_3 ), 
\end{align}	
where the second inequality follows from the assumption that $\beta \le 1/(2\alpha_2 L_{1,G})$. Applying the above inequality recursively gives 
\begin{align*}
    \EE\|y_{t+1}-y_{0,\mu_2}^*\|^2&\leq (1-\beta \lambda_g)^t\|y_0-y_{0,\mu_2}^*\|^2 + \beta^2(\alpha_1 + 2\alpha_2\alpha_3 )\sum_{j=0}^{k-1}(1-\beta\lambda_g)^j\\&\leq (1-\beta \lambda_g)^t\|y_0-y_{0,\mu_2}^*\|^2 + \frac{\beta(\alpha_1 + 2\alpha_2\alpha_3 )}{\lambda_g}.
\end{align*}
Now, assuming $t \ge \frac{1}{\beta \lambda_g} \log \frac{1}{\epsilon}$ (for some $\epsilon>0$), we have
\[
(1-\beta \lambda_g)^t\|y_0-y_{0,\mu_2}^*\|^2 \le e^{-\beta \lambda_g t} \|y_0-y_{0,\mu_2}^*\|^2 \le \epsilon \|y_0-y_{0,\mu_2}^*\|^2,
\]
which together with the assumption that $\beta \le \epsilon \lambda_g$, imply the result.

\subsection{Proof of Lemma~\ref{lemma_sgd}}\label{lemma_sgd:proof}

\begin{proof}
First, note that
\begin{align*}
\EE \left[\|y_{t_k} - y_{\eta_2,\mu_2}^*(x_k) \|^2 \right] &\le 2\EE \left[\|y_{t_k} -y^*(x_k)\|^2 \right]+
2\EE \left[\|y^*(x_k)- y_{\eta_2,\mu_2}^*(x_k)\|^2 \right],\\
\EE \left[\|y_{t_k} - y^*(x_k) \|^2 \right] &\le 2 \EE \left[\|y_{t_k} - y_{0,\mu_2}^*(x_k)  \|^2 \right]
+ 2 \EE \left[\|y^*(x_k) - y_{0,\mu_2}^*(x_k)  \|^2 \right],\\
\EE \left[\|y_0 - y_{0,\mu_2}^*(x_k) \|^2 \right] &\le 2 \EE \left[\|y_0 - y^*(x_k) \|^2 \right]
+ 2 \EE \left[\|y^*(x_k) - y_{0,\mu_2}^*(x_k)  \|^2 \right],
\end{align*}
which in the view of Proposition~\ref{zeroth-order bilevel}.a), imply that
\begin{align*}
\EE \left[\|y^*(x_k)- y_{\eta_2,\mu_2}^*(x_k)\|^2 \right] &\le \frac{2 L_{1,g}}{\lambda_g }\left(\eta_2^2n + \mu_2^2 m\right),\\
\EE \left[\|y_{t_k} - y^*(x_k) \|^2 \right] &\le 2 \EE \left[\|y_{t_k} - y_{0,\mu_2}^*(x_k)  \|^2 \right]
+ \frac{4 L_{1,g} \mu_2^2 m}{\lambda_g },\\
\EE \left[\|y_0 - y_{0,\mu_2}^*(x_k) \|^2 \right] &\le 2 \EE \left[\|y_0 - y^*(x_k) \|^2 \right]
+ \frac{4 L_{1,g} \mu_2^2 m}{\lambda_g }.
\end{align*}
We then obtain \eqnok{y_cnvrg_st} by noting Theorem~\ref{sgd:general} and the choice of $t_k, \beta_t$ in \eqnok{alpha_beta_st}.
\end{proof}

\section{Concluding Remarks}\label{sec:conc}
In this paper, we focus on studying bilevel optimization problems when only noisy evaluations of the inner/outer objectives are available. We first extend the well-known zeroth-order gradient/hessian approximations to the case when two block coordinates are perturbed by different Gaussian random vectors. We recover similar results depending on two smoothing parameters associated with each Gaussian convolution. We then exploit these results in the framework of bilevel optimization and for the first time, establish sample complexity bounds for fully zeroth-order bilevel optimization methods. Examining these results to get the optimal dependence on the problem dimension and target accuracy is left as a future research direction.

\section{Appendix: Supplemental Proofs}\label{proofs}

\subsection{Proof of Proposition \ref{general_prop}}\label{general_prop:proof}
\paragraph{(a).} If $q$ is strongly convex, then for every $t\in[0,1]$:
\begin{align*}
    q(tx+(1-t)x',ty+(1-t)y') \le t q(x,y)+(1-t)q(x',y') - \frac{1}{2} \lambda_q t(1-t) \left(\|x-x'\|^2 + \|y-y'\|^2\right).
\end{align*}
Replacing $x$ and $x'$ with $x+\eta u$, and $x'+\eta u$, also, $y$ and $y'$ with $y+\mu v$, and $y'+\mu v$ in the inequality above, and applying an expectation to both sides gives:  
\begin{align*}
    q_{\eta,\mu}(tx\!+\!(1-t)x',ty\!+\!(1-t)y') &\le t q_{\eta,\mu}(x,y)\!+\!(1-t)q_{\eta,\mu}(x',y') \!- \!\frac{1}{2} \lambda_q t(1-t)\! \left(\|x-x'\|^2\! +\! \|y-y'\|^2\right),
\end{align*}
which establishes the strong convexity of $q_{\eta,\mu}$. Clearly, setting $\lambda_q=0$ in the equations above immediately establishes the result for convex functions. 

\paragraph{(b).} For $q\in \mathcal{C}^0(\mathbb{R}^{n+m};L_{0,q})$, by the Jensen's inequality we have
\begin{align*}
    | q_{\eta,\mu}(x,y) - q_{\eta,\mu}(x',y')| &= | \mathbb{E}_{u,v}\left[ q(x+\eta u, y+\mu v) - q(x'+\eta u, y'+\mu v)\right]|\\ & \leq  \mathbb{E}_{u,v}\left[ \left|q(x+\eta u, y+\mu v) - q(x'+\eta u, y'+\mu v)\right|\right]\\ &\leq \mathbb{E}_{u,v}\left[ L_{0,q} \left\|(x+\eta u, y+\mu v) - (x'+\eta u, y'+\mu v) \right\|  \right]\\& = L_{0,q} \left\|(x,y) - (x',y') \right\|. 
\end{align*}
For $q\in \mathcal{C}^1(\mathbb{R}^{n+m};L_{1,q})$ since the order of expectation and gradient are interchangable as 
\[
\nabla q_{\eta,\mu}(x,y) = \nabla \mathbb{E}_{u,v} [q(x+\eta u, y+\mu v)] =  \mathbb{E}_{u,v} [\nabla q(x+\eta u, y+\mu v)],
\]
a similar argument goes through:
\begin{align*}
    \| \nabla q_{\eta,\mu}(x,y) - \nabla q_{\eta,\mu}(x',y')\| &= \left \| \mathbb{E}_{u,v}\left[ \nabla q(x+\eta u, y+\mu v) - \nabla q(x'+\eta u, y'+\mu v)\right]\right \|\\ & \leq  \mathbb{E}_{u,v} \left\|\nabla q(x+\eta u, y+\mu v) - \nabla q(x'+\eta u, y'+\mu v)\right\|\\ &\leq L_{1,q} \mathbb{E}_{u,v}  \left\|(x+\eta u, y+\mu v) - (x'+\eta u, y'+\mu v) \right\|  \\& = L_{1,q} \left\|(x,y) - (x',y') \right\|. 
\end{align*}
For $q\in \mathcal{C}^2(\mathbb{R}^{n+m};L_{2,q})$, the proof is very similar to the case $q\in \mathcal{C}^1(\mathbb{R}^{n+m};L_{1,q})$ and skipped here.

\subsection{Proof of Proposition \ref{propNestApprox}}\label{propNestApprox:proof}
\paragraph{(a).} For $u\sim\mathcal{N}(0,I_n)$ and $v\sim\mathcal{N}(0,I_m)$:
\begin{align*}
    | q_{\eta,\mu}(x,y) - q(x,y)| &= | \mathbb{E}_{u,v}\left[ q(x+\eta u, y+\mu v) - q(x,y)\right]|\\ & \leq  \mathbb{E}_{u,v}\left[ \left|q(x+\eta u, y+\mu v) - q(x,y)\right|\right]\\ &\leq \mathbb{E}_{u,v}\left[ L_{0,q}\sqrt{\eta^2\|u\|^2 + \mu^2\|v\|^2}\right] \\
    &\leq L_{0,q}\sqrt{\mathbb{E}_{u,v}\left[ \eta^2\|u\|^2 + \mu^2\|v\|^2\right] }\\ &= L_{0,q}\sqrt{\eta^2n + \mu^2 m},
\end{align*}
where in the second line we used the Jensen's inequality, the third line follows from the smoothness assumption, the forth line again uses the Jensen's inequality (noting that the square-root function is concave), and finally in the last equation we used the fact that for $u\sim\mathcal{N}(0,I_n)$, $\mathbb{E}\|u\|^2 = n$. 

\paragraph{(b).} Notice that $q\in \mathcal{C}^1(S;L_{1,q})$ implies that for $z,z'\in\mathbb{R}^{n+m}$,
\begin{equation*}
    \left |q(z) - q(z') - (z-z')^\top \nabla q(z')  \right| \leq \frac{L_{1,q}}{2}\|z-z'\|^2.
\end{equation*}
In view of this, and using a similar line of argument as part (a), we have
\begin{align*}
    | q_{\eta,\mu}(x,y) - q(x,y)| &= \left | \mathbb{E}_{u,v}\left[ q(x+\eta u, y+\mu v) - q(x,y) - (\eta u,\mu v)^\top \nabla q(x,y) \right]\right |\\ & \leq  \mathbb{E}_{u,v}\left[ \left|q(x+\eta u, y+\mu v) - q(x,y) - (\eta u,\mu v)^\top \nabla q(x,y)\right|\right]\\ &\leq \mathbb{E}_{u,v}\left[ \frac{L_{1,q}}{2}\left(\eta^2\|u\|^2 + \mu^2\|v\|^2\right)\right] \\
    &= \frac{L_{1,q}}{2}\left(\eta^2n + \mu^2 m\right).
\end{align*}
For the next part, notice that
\begin{align}\notag 
\nabla_x q_{\eta,\mu}(x,y) - \nabla_x q(x,y)  &= \mathbb{E}_{u,v} \left[  \frac{q(x+\eta u, y+\mu v) - \eta u^\top  \nabla_x q(x,y)}{\eta} \ u \right]\\ \notag &= \mathbb{E}_{u,v} \left[  \frac{q(x+\eta u, y+\mu v) - q(x, y) - \eta u^\top  \nabla_x q(x,y) }{\eta}  u\right] \\ \notag &= \mathbb{E}_{u,v} \left[  \frac{q(x+\eta u, y+\mu v) - q(x, y) - \eta u^\top  \nabla_x q(x,y) -\mu v^\top  \nabla_y q(x,y) }{\eta}  u\right]\\  &= \mathbb{E}_{u,v} \left[  \frac{q(x+\eta u, y+\mu v) - q(x, y) - (\eta u,\mu v)^\top \nabla q(x,y)  }{\eta}  u\right].
\end{align}
Using the Jensen's inequality, this immediately implies that
\begin{align*}
\|\nabla_x q_{\eta,\mu}(x,y) - \nabla_x q(x,y)\|  &\leq  \mathbb{E}_{u,v}\left[ \left\|    \frac{q(x+\eta u, y+\mu v) - q(x, y) - (\eta u,\mu v)^\top \nabla q(x,y)  }{\eta}  u \right\|\right]\\ &= \mathbb{E}_{u,v}\left[ \left|    \frac{q(x+\eta u, y+\mu v) - q(x, y) - (\eta u,\mu v)^\top \nabla q(x,y)  }{\eta}   \right|\|u\|\right]\\&\leq \mathbb{E}_{u,v}\left[ \frac{L_{1,q}}{2\eta}\left(\eta^2\|u\|^2 + \mu^2\|v\|^2\right)\|u\|\right]\\&\leq \frac{L_{1,q}}{2\eta}\left(\eta^2(n+3)^{\frac{3}{2}} + \mu^2 m n^{\frac{1}{2}}\right)
\end{align*}
where we used the facts that for $u\sim\mathcal{N}(0,I_n)$, $\mathbb{E}\|u\| \leq n^{\frac{1}{2}}$, $\mathbb{E}\|u\|^2 = n$ , and $\mathbb{E}\|u\|^3 \leq (n+3)^{\frac{3}{2}}$ (see Lemma \ref{lemmaExpectationNormPower}). In a similar fashion, we can show that 
\begin{align*}
\|\nabla_y q_{\eta,\mu}(x,y) - \nabla_y q(x,y)\| \leq  \frac{L_{1,q}}{2\mu}\left(\mu^2(m+3)^{\frac{3}{2}} + \eta^2 n m^{\frac{1}{2}}\right),
\end{align*}
and as a result
\begin{align}
\|\nabla q_{\eta,\mu}(x,y) - \nabla q(x,y)\| \leq \frac{L_{1,q}}{2}\left(\eta(n+3)^{\frac{3}{2}} + \frac{\mu^2}{\eta} m n^{\frac{1}{2}} + \frac{\eta^2}{\mu} n m^{\frac{1}{2}}+ \mu(m+3)^{\frac{3}{2}} \right).\label{eqU1x}
\end{align}
\begin{remark}\label{remarkLeibniz}
    Since a Gaussian random variable is unbounded, and the calculation of the expectation undergoes an improper integration, standard Leibniz-type conditions cannot warrant the exchange of expectation and differentiation orders. However, under stricter condition on $q$ (e.g., see \cite{protter1985differentiation}), that such exchange of orders becomes possible, the bound in \eqref{eqU1x} can be significantly improved as 
\begin{align}\notag 
\nabla q_{\eta,\mu}(x,y) - \nabla q(x,y)  &= \nabla \mathbb{E}_{u,v} \left[  q(x+\eta u, y+\mu v) -  q(x,y) \right]\\& = \mathbb{E}_{u,v} \left[  \nabla q(x+\eta u, y+\mu v) -  \nabla q(x,y) \right].\label{nablaxDiff}
\end{align}
Then, using the Jensen's inequality, this immediately implies that
\begin{align*}
\|\nabla q_{\eta,\mu}(x,y) - \nabla q(x,y)\|  &\leq  \mathbb{E}_{u,v}\left[ \left\|  \nabla q(x+\eta u, y+\mu v) -  \nabla q(x,y)   \right\|\right]\\&\leq \mathbb{E}_{u,v}\left[ L_{1,q}\sqrt{\eta^2\|u\|^2 + \mu^2\|v\|^2}\right] \\
    &\leq L_{1,q}\sqrt{\mathbb{E}_{u,v}\left[ \eta^2\|u\|^2 + \mu^2\|v\|^2\right] }\\ &= L_{1,q}\sqrt{\eta^2n + \mu^2 m}.
\end{align*}
\end{remark}

\paragraph{(c).} As stated before,  $q\in \mathcal{C}^2(S;L_{2,q})$ implies that for $z,z'\in\mathbb{R}^{n+m}$, 
\begin{equation}\label{secondOrder}
   \left|q(z') - q(z) -\langle \nabla q(z), z'-z\rangle- \frac{1}{2}\langle \nabla^2 q(z)(z'-z), z'-z\rangle\right |\leq \frac{L_{2,q}}{6}\|z-z'\|^3.
\end{equation}
On the other hand,
\begin{align}\notag 
\nabla_x q_{\eta,\mu}(x,y) - \nabla_x q(x,y)  &= \mathbb{E}_{u,v} \left[  \frac{q(x+\eta u, y+\mu v) - \eta u^\top  \nabla_x q(x,y)}{\eta} \ u \right]\\ \notag &= \mathbb{E}_{u,v} \left[  \frac{q(x+\eta u, y+\mu v) - q(x, y) - \eta u^\top  \nabla_x q(x,y) }{\eta}  u\right] \\ \notag &= \mathbb{E}_{u,v} \left[  \frac{q(x+\eta u, y+\mu v) - q(x, y) - \eta u^\top  \nabla_x q(x,y) -\mu v^\top  \nabla_y q(x,y) }{\eta}  u\right]\\  &= \mathbb{E}_{u,v} \left[  \frac{q(x+\eta u, y+\mu v) - q(x, y) - (\eta u,\mu v)^\top \nabla q(x,y)  }{\eta}  u\right].\label{nablaxDiff}
\end{align}
Given $x$ and $y$, we define the function
\begin{equation}
    e_{x,y}(\delta ,\delta') = q(x+\delta , y+\delta') - q(x, y) - (\delta ,\delta')^\top \nabla q(x,y) - \frac{1}{2} (\delta,\delta')^\top \nabla^2 q(x,y)(\delta,\delta').\label{eqexy}
\end{equation}
Averaging \eqref{nablaxDiff} with a version of itself evaluated for $-u$ and $-v$ gives
\begin{align*}\notag 
\nabla_x q_{\eta,\mu}(x,y) - \nabla_x q(x,y) & = \mathbb{E}_{u,v} \left[  \frac{q(x+\eta u, y+\mu v) - q(x-\eta u, y-\mu v) - 2(\eta u,\mu v)^\top \nabla q(x,y)  }{2\eta}  u\right]\\ & = \mathbb{E}_{u,v} \left[  \frac{e_{x,y}(\eta u,\mu v) - e_{x,y}(-\eta u,-\mu v) }{2\eta}  u\right].
\end{align*}
Notice that
\begin{align}\notag 
|e_{x,y}(\eta u,\mu v) - e_{x,y}(-\eta u,-\mu v)|&\leq |e_{x,y}(\eta u,\mu v) | + |e_{x,y}(-\eta u,-\mu v)|\\ \notag  &\leq \frac{L_{2,q}}{3}\|(\eta u,\mu v)\|^3\\ \notag  &\leq \frac{L_{2,q}}{3} \left(\|\eta u\| + \|\mu v\|\right)^3\\&\leq  \frac{4 L_{2,q}}{3} \left(\|\eta u\|^3 + \|\mu v\|^3\right),\label{exy:bound}
\end{align}
where the second inequality is a result of \eqref{secondOrder}, the third inequality is an application of the triangle inequality for vectors $(\eta u,0)$ and $(0,\mu y)$, and the last inequality is an implication of $(a+b)^k\leq 2^{k-1}(a^k + b^k)$ which holds for $a,b\geq 0$ and $k\geq 1$. We can now bound the difference norm as
\begin{align*}\notag 
\|\nabla_x q_{\eta,\mu}(x,y) - \nabla_x q(x,y)\| & = \left\| \mathbb{E}_{u,v} \left[  \frac{e_{x,y}(\eta u,\mu v) - e_{x,y}(-\eta u,-\mu v) }{2\eta}  u\right] \right\|\\ &\leq  \mathbb{E}_{u,v} \left[ \left\| \frac{e_{x,y}(\eta u,\mu v) - e_{x,y}(-\eta u,-\mu v) }{2\eta}  u \right\|\right]\\ & = \mathbb{E}_{u,v} \left[ \left | \frac{e_{x,y}(\eta u,\mu v) - e_{x,y}(-\eta u,-\mu v) }{2\eta}\right|\left\|  u \right\|\right]\\& \leq \frac{2L_{2,q}}{3\eta }\mathbb{E}_{u,v} \left[ \left ( \eta^3\|u\|^3+\mu^3\|v\|^3\right)\left\|  u \right\|\right]\\&\leq \frac{2L_{2,q}}{3}\left(\eta^2(n+4)^2 + \frac{\mu^3}{\eta} (m+3)^{\frac{3}{2}}n^{\frac{1}{2}} \right),
\end{align*}
where for the last inequality we employed Lemma \ref{lemmaExpectationNormPower}, specifically, $\mathbb{E}\|u\| \leq n^{\frac{1}{2}}$, and for $p\geq 2$, $\mathbb{E}\|u\|^p \leq (n+p)^{\frac{p}{2}}$.
In a similar fashion, we can show that 
\begin{align*}
\|\nabla_y q_{\eta,\mu}(x,y) - \nabla_y q(x,y)\| \leq \frac{2L_{2,q}}{3}\left(\mu^2(m+4)^2 + \frac{\eta^3}{\mu} (n+3)^{\frac{3}{2}}m^{\frac{1}{2}} \right),
\end{align*}
and as a result
\begin{align*}
\|\nabla q_{\eta,\mu}(x,y) - \nabla q(x,y)\| \leq \frac{2L_{2,q}}{3}\left(\eta^2(n+4)^2 + \frac{\mu^3}{\eta} (m+3)^{\frac{3}{2}}n^{\frac{1}{2}}  + \frac{\eta^3}{\mu} (n+3)^{\frac{3}{2}}m^{\frac{1}{2}} +  \mu^2(m+4)^2  \right).
\end{align*}
For the next part, and to bound the difference between $\nabla^2 q_{\eta,\mu}(x,y)$ and $\nabla^2 q(x,y)$, defining $\delta(x,y) = q_{\eta,\mu}(x,y) - q(x,y)$, we have
\begin{align}\notag 
    \|\nabla^2\delta\| = \left\|\begin{bmatrix}\nabla_{xx}^2\delta & \nabla_{xy}^2\delta\\ \nabla_{yx}^2\delta & \nabla_{yy}^2\delta\end{bmatrix} \right\| &= \left\|\begin{bmatrix}\nabla_{xx}^2\delta & 0\\ 0& 0\end{bmatrix} + \begin{bmatrix}0 & 0\\ 0& \nabla_{yy}^2\delta\end{bmatrix} + \begin{bmatrix}0 & \nabla_{xy}^2\delta\\ \nabla_{xy}^2\delta^\top& 0\end{bmatrix}\right\| \\ &\leq \|\nabla_{xx}^2\delta\| + \|\nabla_{yy}^2\delta\| + \|\nabla_{xy}^2\delta\|,\label{BlockOperatorNorm}
\end{align}
where we used the triangle inequality, plus the facts that zero padding does not change the operator norm, and an anti-block diagonal matrix with blocks $\nabla_{xy}^2\delta$ and $\nabla_{xy}^2\delta^\top$ has the same operator norm as $\nabla_{xy}^2\delta$. This allows us to bound $ \nabla^2 q_{\eta,\mu}(x,y) - \nabla^2 q(x,y)$ by bounding its blocks. We start by the $\nabla^2_{xx}$ block as follows:
\begin{align*}
    \nabla^2_{xx}q_{\eta,\mu}(x,y) =& ~\EE_{u,v}\left[ \left(uu^\top - I_n\right)\frac{q(x+\eta u,y+\mu v)}{\eta^2}\right]\\ =&~\EE_{u,v}\left[ \left(uu^\top - I_n\right)\frac{q(x+\eta u,y+{\mu v}) - q(x,y)  }{\eta^2}\right]\\
    =&~\EE_{u,v}\left[ \left(uu^\top - I_n\right)\frac{e_{x,y}(\eta u,\mu v) }{\eta^2}\right] + \EE_{u,v}\left[ \left(uu^\top - I_n\right)\frac{  u^\top \nabla_x q(x,y)  }{\eta}\right]\\&+ 
    \frac{1}{2}\EE_{u,v}\left[ \left(uu^\top - I_n\right)   u^\top \nabla^2_{xx} q(x,y)u \right],
\end{align*}
where the second equality is thanks to $\EE_u[(uu^\top-I_n)]=0$, and $e_{x,y}$ follows the notation in \eqref{eqexy}. We now simplify each term in the third equality. For the second term we have
\begin{align*}
   \frac{1}{\eta} \EE_{u,v}\left[ \left(uu^\top - I_n\right)  u^\top \nabla_x q(x,y)  \right] &=\frac{1}{\eta} \EE_{u}\left[ u \nabla_x q(x,y)^\top uu^\top     \right]\\&=0.
\end{align*}
Here we used the result (see \S 8.2.3 of \cite{petersen2008matrix}) that for a normal random variable $a\sim\mathcal{N}(\alpha,I)$ and fixed vector $b$: 
\[
\EE\left[ ab^\top aa^\top  \right] = \alpha b^\top\left(I+\alpha\alpha^\top \right)+\left(I+\alpha\alpha^\top \right)b\alpha^\top + b^\top\alpha \left(I-\alpha\alpha^\top \right).
\]
For the third term we have
\begin{align*}
    \frac{1}{2}\EE_{u,v}\left[ \left(uu^\top - I_n\right)   u^\top \nabla^2_{xx} q(x,y)u \right] =& ~ \frac{1}{2}\EE_{u}\left[ uu^\top \nabla^2_{xx} q(x,y)  u u^\top  \right] - \frac{1}{2}\tr\left(\nabla^2_{xx} q(x,y) \right)I_n\\ =& ~\nabla^2_{xx} q(x,y).
\end{align*}
Here we used the result (see \S 8.2.4 of \cite{petersen2008matrix}) that for a normal random variable $a\sim\mathcal{N}(\alpha,I)$ and fixed matrix $B$: 
\begin{align*}
    \EE\left[ aa^\top Baa^\top  \right] = \left(I+\alpha\alpha^\top \right)\left(B+B^\top\right)\left(I+\alpha\alpha^\top \right)+\alpha^\top B \alpha \left(I-\alpha\alpha^\top \right) + \tr(B)\left(I+\alpha\alpha^\top \right).
\end{align*}
This implies that
\begin{align} \notag 
   \left\| \nabla^2_{xx}q_{\eta,\mu}(x,y) - \nabla^2_{xx}q(x,y)\right\| &=\left\|\EE_{u,v}\left[ \left(uu^\top - I_n\right)\frac{e_{x,y}(\eta u,\mu v) }{\eta^2}\right] \right\|\\ \notag  &\leq \EE_{u,v}\left[ \left\| \left(uu^\top - I_n\right)\frac{e_{x,y}(\eta u,\mu v) }{\eta^2} \right\|\right]\\ \notag  & =  \EE_{u,v}\left[ \left | \frac{e_{x,y}(\eta u,\mu v)}{\eta^2}\right| \left\|uu^\top - I_n\right\|  \right]\\& \notag \leq \frac{2 L_{2,q}}{3\eta^2} \EE_{u,v}\left[\left(\|\eta u\|^3 + \|\mu v\|^3\right)\left\|uu^\top - I_n\right\|\right]\\ \notag & \leq \frac{2 L_{2,q}}{3\eta^2} \EE_{u,v}\left[\left(\|\eta u\|^3 + \|\mu v\|^3\right)\left(\|u\|^2 +1\right)\right]
   \\&\leq \frac{2 L_{2,q}}{3}\left(\eta (n+5)^{5/2} + \eta(n+3)^{3/2} + \frac{\mu^3}{\eta^2}(n+1)(m+3)^{3/2}  \right), \label{eqU2xx}
\end{align}
where for the second inequality a similar line of argument as \eqref{exy:bound} is used. In a similar fashion, we can show that
\begin{align}
   \left\| \nabla^2_{yy}q_{\eta,\mu}(x,y) - \nabla^2_{yy}q(x,y)\right\| \leq \frac{2 L_{2,q}}{3}\left(\mu (m+5)^{5/2} + \mu(m+3)^{3/2} + \frac{\eta^3}{\mu^2}(m+1)(n+3)^{3/2}  \right). \label{eqU2yy}
\end{align}
Finally, for the $\nabla_{xy}$ component we get
\begin{align*}
    \nabla^2_{xy}q_{\eta,\mu}(x,y) =& ~\EE_{u,v}\left[ uv^\top \frac{q(x+\eta u,y+\mu v)}{\eta\mu}\right]\\ =&~\EE_{u,v}\left[ uv^\top \frac{q(x+\eta u,y+{\mu v}) - q(x,y)  }{\eta\mu}\right]\\
    =&~\EE_{u,v}\left[ uv^\top\frac{e_{x,y}(\eta u,\mu v) }{\eta\mu}\right] + \EE_{u,v}\left[ uv^\top \frac{  (\eta u,\mu v)^\top \nabla q(x,y)  }{\eta\mu}\right]\\&+ 
    \frac{1}{2}\EE_{u,v}\left[ uv^\top  \frac{(\eta u,\mu v)^\top \nabla^2 q(x,y) (\eta u,\mu v)}{\eta\mu } \right],
\end{align*}
Each term in the third equality simplifies as follows. For the second term we have
\begin{align*}
   \frac{1}{\eta\mu } \EE_{u,v}\left[ uv^\top  (\eta u,\mu v)^\top \nabla q(x,y) \right] &= \EE_{u,v}\left[ uv^\top  \left(\eta u^\top \nabla_x q(x,y) +  \mu v^\top \nabla_y q(x,y)\right) \right]\\&=0.
\end{align*}
For the third term we have
\begin{align*}
    \frac{1}{2 \eta\mu }&\EE_{u,v}\left[ uv^\top  (\eta u,\mu v)^\top \nabla^2 q(x,y) (\eta u,\mu v) \right] \\&=  \frac{1}{2 \eta\mu }\EE_{u,v}\left[ uv^\top  \left( \eta^2 u^\top\nabla^2_{xx} q(x,y) u + \mu^2v^\top \nabla^2_{yy}q(x,y) v+ 2\eta \mu u^\top \nabla^2_{xy}q(x,y)v\right) \right]\\& = \EE_u\left[uu^\top\right]\nabla_{x,y}q(x,y)\EE_v\left[vv^\top\right]\\& = \nabla_{x,y}q(x,y).
\end{align*}
This result implies that
\begin{align}\notag 
   \left\| \nabla^2_{xy}q_{\eta,\mu}(x,y) - \nabla^2_{xy}q(x,y)\right\| &=\left\|\EE_{u,v}\left[ uv^\top\frac{e_{x,y}(\eta u,\mu v) }{\eta\mu}\right] \right\|\\\notag  &\leq \EE_{u,v}\left[ \left\| uv^\top\frac{e_{x,y}(\eta u,\mu v) }{\eta\mu} \right\|\right]\\ \notag  & =  \EE_{u,v}\left[ \left | \frac{e_{x,y}(\eta u,\mu v)}{\eta\mu }\right| \left\|uv^\top \right\|  \right]\\ \notag  &\leq \frac{2 L_{2,q}}{3\eta\mu} \EE_{u,v}\left[\left(\|\eta u\|^3 + \|\mu v\|^3\right)\left\|uv^\top \right\|\right]\\ \notag  & \leq \frac{2 L_{2,q}}{3\eta\mu} \EE_{u,v}\left[\left(\|\eta u\|^3 + \|\mu v\|^3\right)\|u\|\|v\| \right]
   \\&\leq \frac{2 L_{2,q}}{3}\left(\frac{\eta^2}{\mu}(n+4)^2 m^{1/2} + \frac{\mu^2}{\eta}(m+4)^2 n^{1/2} \right).\label{eqU2xy}
\end{align}
Using \eqref{eqU2xy}, \eqref{eqU2yy} and \eqref{eqU2xx} in \eqref{BlockOperatorNorm}, delivers the advertised result:
\begin{equation}
  \|  \nabla^2 q_{\eta,\mu}(x,y) - \nabla^2 q(x,y)\| \leq U^{(2)}_{xx} + U^{(2)}_{xy} + U^{(2)}_{yy}.\label{eqHessApproxBound}  
\end{equation}
\begin{remark}
    Similar to Remark \ref{remarkLeibniz}, if the conditions to exchange the order of expectation and Hessian hold for $q$, then a similar line of argument as Remark \ref{remarkLeibniz} can improve the bound \eqref{eqHessApproxBound} to
\begin{align*}
\|\nabla^2 q_{\eta,\mu}(x,y) - \nabla^2 q(x,y)\|  \leq  L_{2,q}\sqrt{\eta^2n + \mu^2 m}.
\end{align*}
\end{remark}

\subsection{Proof of Proposition \ref{propNestApprox_stch}}\label{propNestApprox_stch:proof}
\begin{proof}
First, note that by the triangle inequality and by Proposition \ref{propNestApprox}(b) we have
\begin{align*}
\|\nabla_x Q_{\eta,\mu}(x,y,\zeta) - \nabla_x q_{\eta,\mu}(x,y)\|
 &\le  \|\nabla_x Q_{\eta,\mu}(x,y,\zeta) - \nabla_x Q(x,y,\zeta)\| +\|\nabla_x Q(x,y,\zeta) - \nabla_x q(x,y)\|\\
 &+\|\nabla_x q(x,y)-\nabla_x q_{\eta,\mu}(x,y) \|\\
 &\le U^{(1)}_x \left(1+\frac{L_{1,Q}}{L_{1,q}}\right) +\|\nabla_x Q(x,y,\zeta) - \nabla_x q(x,y)\|.
\end{align*}
Taking expectation of the squared of both sides and under Assumption 1, we obtain the first inequality in part a). Similarly, we can obtain  the second inequality and thus, we skip the details.

Now, assuming $q\in \mathcal{C}^2(\mathbb{R}^{n+m};L_{2,q})$, and following a similar approach, we have
\begin{align*}
\|\nabla_{xy} Q_{\eta,\mu}(x,y,\zeta) - \nabla_{xy} q_{\eta,\mu}(x,y)\|
 &\le  \|\nabla_{xy} Q_{\eta,\mu}(x,y,\zeta) - \nabla_{xy} Q(x,y,\zeta)\| +\|\nabla_{xy} Q(x,y,\zeta) - \nabla_{xy} q(x,y)\|\\
 &+\|\nabla_{xy} q(x,y)-\nabla_{xy} q_{\eta,\mu}(x,y) \|\\
 &\le U^{(2)}_{xy} \left(1+\frac{L_{2,Q}}{L_{2,q}}\right) +\|\nabla_{xy} Q(x,y,\zeta) - \nabla_{xy} q(x,y)\|,
\end{align*}
which implies part b).

\end{proof}

\subsection{Proof of Proposition \ref{propMomentBounds}}\label{propMomentBounds:proof}
Notice that
\begin{align*}
    \mathbb{E}\|\tilde \nabla_x q_{\eta,\mu}(x,y) \|^2 &= \frac{1}{\eta^2}\mathbb{E}\left[\left( q(x+\eta u, y+\mu v) - q(x, y)\right)^2\|u\|^2 \right].
\end{align*}
The function-dependent term can be bounded as
\begin{align*}
    ( q(x\!+\!\eta u,& y\!+\!\mu v) - q(x, y))^2 \\&= \left( q(x\!+\!\eta u, y\!+\!\mu v) - q(x, y) - (\eta u,\mu v)^\top \nabla  q(x, y) + (\eta u,\mu v)^\top \nabla  q(x, y) \right)^2\\ &\leq 2\left( q(x\!+\!\eta u, y\!+\!\mu v) - q(x, y) - (\eta u,\mu v)^\top \nabla  q(x, y)\right)^2 + 2\left((\eta u,\mu v)^\top \nabla  q(x, y)\right)^2\\&\leq 
    \frac{L_{1,q}^2}{2}\left(\eta^2\|u\|^2+\mu^2\|v\|^2 \right)^2+ 2\left( \eta u^\top \nabla_x  q(x, y)+ \mu v^\top \nabla_y  q(x, y) \right)^2\\ 
    &\leq L_{1,q}^2\left(\eta^4\|u\|^4+\mu^4\|v\|^4 \right)+ 4\left(\eta^2 \left(u^\top \nabla_x  q(x, y)\right)^2 + \mu^2\left( v^\top \nabla_y  q(x, y) \right)^2\right),
\end{align*}
   which implies
\begin{align}\label{eq:expansionGradq}
      \|\tilde \nabla_x q_{\eta,\mu}(x,y) \|^2&\leq L_{1,q}^2\left[\eta^2\|u\|^6+\frac{\mu^4}{\eta^2}\|v\|^4\|u\|^2 \right]+ 4\left( \left(u^\top \nabla_x  q(x, y)\right)^2 + \frac{\mu^2}{\eta^2}\left( v^\top \nabla_y  q(x, y) \right)^2\right)u^\top u.
\end{align}
Using Lemma \ref{lemmaExpectationNormPower}, the expectation of the first term on the right side of \eqref{eq:expansionGradq} can be bounded as
\begin{align}
       L_{1,q}^2\EE\left[\eta^2\|u\|^6+\frac{\mu^4}{\eta^2}\|v\|^4\|u\|^2 \right] \leq L_{1,q}^2\left(\eta^2(n+6)^3+\frac{\mu^4}{\eta^2}n(m+4)^2 \right). \label{eq:expansionGradq2}
\end{align}
The second term expectation in \eqref{eq:expansionGradq} can be written as
\begin{align}\notag 
   \EE\bigg[ \Big(u^\top \nabla_x  q(x, y)\Big)^2 u^\top u &+ \frac{\mu^2}{\eta^2}\Big( v^\top \nabla_y  q(x, y) \Big)^2 u^\top u \bigg] \\&= \EE\left[u^\top \nabla_x  q(x, y) \nabla_x  q(x, y)^\top u u^\top u\right] + \frac{\mu^2}{\eta^2}\EE\left[v^\top \nabla_y  q(x, y) \nabla_y  q(x, y)^\top v u^\top u\right]\notag \\& = (n+2)\|\nabla_x  q(x, y)\|^2 + \frac{n\mu^2}{\eta^2}\|\nabla_y  q(x, y)\|^2,\label{eq:expansionGradq3}
\end{align}
where in the second equality we used the 4-th moment result  that for symmetric matrices $A$ and $B$ (see \S 8.2.4 of \cite{petersen2008matrix}):
\begin{equation}\label{cookbook:result:4thorder}
\EE u^\top A u u^\top B u = 2\tr(AB) + \tr(A)\tr(B).
\end{equation}
Using \eqref{eq:expansionGradq3} and \eqref{eq:expansionGradq2} in \eqref{eq:expansionGradq} yields 
\[
\EE\left[  \|\tilde \nabla_x q_{\eta,\mu}(x,y) \|^2 \right]\! \leq L_{1,q}^2\left(\eta^2(n+6)^3+\frac{\mu^4}{\eta^2}n(m+4)^2 \right) +  4(n+2) \|\nabla_x  q(x, y)\|^2 +  \frac{4\mu^2}{\eta^2}n\|\nabla_y  q(x, y)\|^2.
\]
   
\subsection{Proof of Proposition \ref{hessMomentBounds}}\label{hessMomentBounds:proof}

We first show part a). Under the formulation in \eqref{stochHessxx} and the fact that $\theta^\top u u^\top\theta = |\theta^\top u|^2\geq 0$,
\begin{align}\notag 
    \left\|  \tilde \nabla_{xx}^2 q_{\eta,\mu}(x,y) \theta  \right\|^2 &= \frac{\left(q(x+\eta u, y+\mu v)\! +\! q(x-\eta u,y-\mu v)\!-\! 2q(x,y )\right)^2}{4\eta^4}\theta^\top\!\left(\|u\|^2 uu^\top - 2uu^\top + I \right)\! \theta \\ &\leq \frac{\left(q(x+\eta u, y+\mu v)\! + \! q(x-\eta u,y-\mu v)\!-\! 2q(x,y )\right)^2}{4\eta^4}\theta^\top\!\left(\|u\|^2 uu^\top + I \right)\! \theta. \label{hessMomArg}
\end{align}
Using the definition of $e_{x,y}$ in \eqref{eqexy}, we can bound the function-dependent term as
\begin{align}\notag 
    ((q(x+\eta u, y+\mu v)\! & + \! q(x-\eta u,y-\mu v)\!-\! 2q(x,y ))^2 \\\notag  & = \left( e_{x,y}(\eta u,\mu v) + e_{x,y}(-\eta u,-\mu v) + \langle (\eta u,\mu v), \nabla^2 q(x,y) (\eta u,\mu v)\rangle \right)^2\\\notag  &\leq 2\left( e_{x,y}(\eta u,\mu v) + e_{x,y}(-\eta u,-\mu v)\right)^2 + 2\langle (\eta u,\mu v), \nabla^2 q(x,y) (\eta u,\mu v)\rangle ^2 \\ &\leq 8 L_{2,q}^2\left(\|\eta u\|^6 + \|\mu v\|^6\right)+ 2\langle (\eta u,\mu v), \nabla^2 q(x,y) (\eta u,\mu v)\rangle ^2,\label{functDependentsq}
\end{align}
where in the last inequality we used a similar argument as \eqref{exy:bound}, to get:
\begin{align*}
\left(e_{x,y}(\eta u,\mu v) + e_{x,y}(-\eta u,-\mu v)\right)^2&\leq \left(|e_{x,y}(\eta u,\mu v) | + |e_{x,y}(-\eta u,-\mu v)|\right)^2\\ &\leq  \frac{16 L_{2,q}^2}{9} \left(\|\eta u\|^3 + \|\mu v\|^3\right)^2\\ 
&\leq 4 L_{2,q}^2 \left(\|\eta u\|^6 + \|\mu v\|^6\right).
\end{align*}
On the other hand,
\begin{align*}
    \langle (\eta u,\mu v), \nabla^2 q(x,y) (\eta u,\mu v)\rangle ^2 &\!=\! \left( \eta^2 u^\top\nabla^2_{xx} q(x,y) u + 2\eta\mu u^\top\nabla^2_{xy} q(x,y) v + \mu^2 v^\top\nabla^2_{yy} q(x,y) v\right)^2\\ &\leq 3\eta^4\langle u,\! \nabla^2_{xx} q(x,y) u\rangle^2 \!+ \! 6\eta^2 \mu^2 \langle u,\! \nabla^2_{xy}q(x,y) v \rangle^2 \!+\! 3\mu^4 \langle v,\! \nabla^2_{yy}q(x,y) v \rangle^2,
\end{align*}
where we used $(a+b+c)^2\leq 3(a^2+b^2+c^2)$, which is simply an instance of the Cauchy–Schwarz inequality. 
Dropping the dependency of $q$ to $(x,y)$ for brevity, this result together with \eqref{functDependentsq} and \eqref{hessMomArg} implies that 
\begin{align}\notag 
    \left\|  \tilde \nabla_{xx}^2 q_{\eta,\mu} \theta  \right\|^2 \leq & \underbrace{2L_{2,q}^2\left( \eta^2\|u\|^6 + \frac{\mu^6}{\eta^4}\|v\|^6\right)\theta^\top \left(\|u\|^2 uu^\top + I \right) \theta}_\text{$T_1(u,v)$}\\ \notag  & + \underbrace{\frac{3}{2} \langle u,\! \nabla^2_{xx} q~\! u\rangle^2 \theta^\top \left(\|u\|^2 uu^\top + I \right) \theta}_\text{$T_2(u)$} \\ \notag  & + \underbrace{\frac{3\mu^2}{\eta^2}  \langle u,\! \nabla^2_{xy}q~\! v \rangle^2 \theta^\top \left(\|u\|^2 uu^\top + I \right) \theta}_\text{$T_3(u,v)$} \\ & + \underbrace{\frac{3\mu^4}{2\eta^4} \langle v,\! \nabla^2_{yy}q ~\! v \rangle^2 \theta^\top \left(\|u\|^2 uu^\top + I \right) \theta}_\text{$T_4(u,v)$}. \label{T1-T4Expansion}
\end{align}
In the sequel we bound the expectation of $T_1(u,v)$ through $T_4(u,v)$. For the first term we have
\begin{align*}
    \EE\left[T_1(u,v)\right] = 2L_{2,q}^2\EE\left[ \eta^2\|u\|^8(\theta^\top u)^2  + \eta^2 \|u\|^6\|\theta\|^2 + \frac{\mu^6}{\eta^4}\|v\|^6 \|u\|^2 (\theta^\top u)^2 +  \frac{\mu^6}{\eta^4}\|v\|^6 \|\theta\|^2 \right].
\end{align*}
Appealing to \eqref{cookbook:result:4thorder} we have $\EE  \|u\|^2 (\theta^\top u)^2 = (n+2)\|\theta\|^2$, and for an even integer $k\geq 4$, using the Cauchy-Shwartz inequality we get
\begin{align}\notag 
\EE  \|u\|^k (\theta^\top u)^2 &\leq \left( \EE \|u\|^{2k}\right)^{1/2} \left( \EE (\theta^\top u)^4\right)^{1/2}\\ &\leq \sqrt{3}(n+2k)^{k/2}\|\theta \|^2,\label{k-momentTimesNorm}
\end{align}
where for the first term of product we used Lemma \ref{lemmaExpectationNormPower}, and for the second term we used 
\eqref{cookbook:result:4thorder} as
\[
\EE (\theta^\top u)^4 = \EE u^\top \theta\theta^\top u u^\top \theta\theta^\top u = 2\tr(\theta\theta^\top \theta\theta^\top) + \tr(\theta\theta^\top)^2 = 3\|\theta\|^4.
\]
Using \eqref{k-momentTimesNorm} and Lemma \ref{lemmaExpectationNormPower} we can bound $\EE\left[T_1(u,v)\right]$ as follows:
\begin{align}\notag
    \EE\left[T_1(u,v)\right] &\leq  2L_{2,q}^2\left( \sqrt{3}\eta^2(n+16)^4  + \eta^2 (n+6)^3 + \frac{\mu^6}{\eta^4}(m+6)^3 (n+2) +  \frac{\mu^6}{\eta^4}(m+6)^3  \right)\|\theta\|^2\\ &\leq  2L_{2,q}^2\left( 2\eta^2(n+16)^4 + \frac{\mu^6}{\eta^4}(m+6)^3 (n+3)\right)\|\theta\|^2.\label{T1:Bound}
\end{align}
For $T_2(u)$ we have
\begin{align}\label{T2:expand}
    \EE\left[ T_2(u) \right] &= \frac{3}{2} \EE\left[\langle u,\! \nabla^2_{xx} q~\! u\rangle^2(\theta^\top u)^2\|u\|^2 \right] + \frac{3}{2}\EE\left[ \langle u,\! \nabla^2_{xx} q~\! u\rangle^2\right]\|\theta\|^2.
\end{align}
We start by bounding the first term in \eqref{T2:expand}. 
To do this, using the definition of the expectation for any $\tau\in(0,1)$, we have
\begin{align*}
\EE \Big[\langle u,\! \nabla^2_{xx} q~\!& u \rangle^2(\theta^\top u)^2\|u\|^2 \Big]
\\&=\frac{1}{(2\pi)^{n/2}}\int_{\mathbb{R}^n} \langle u,\! \nabla^2_{xx} q~\! u \rangle^2(\theta^\top u)^2\|u\|^2 \exp(-\frac{\|u\|^2}{2})du\\ &=
\frac{1}{(2\pi)^{n/2}}\int_{\mathbb{R}^n} \|u\|^2\exp(-\frac{\tau\|u\|^2}{2}) \langle u,\! \nabla^2_{xx} q~\! u \rangle^2(\theta^\top u)^2 \exp(-\frac{(1-\tau)\|u\|^2}{2})du\\
&\leq \frac{2}{(2\pi)^{n/2}e\tau}\int_{\mathbb{R}^n}  \langle u,\! \nabla^2_{xx} q~\! u \rangle^2 (\theta^\top u)^2 \exp(-\frac{(1-\tau)\|u\|^2}{2})du\\&
= \frac{2}{(2\pi)^{n/2}e\tau(1-\tau)^{3+n/2}}\int_{\mathbb{R}^n}  \langle u,\! \nabla^2_{xx} q~\! u \rangle^2 (\theta^\top u)^2 \exp(-\frac{\|u\|^2}{2})du\\&
=\frac{2}{e\tau(1-\tau)^{3+n/2}}\EE \left[\langle u,\! \nabla^2_{xx} q~\! u \rangle^2(\theta^\top u)^2\right], 
\end{align*}
where the inequality is thanks to the relation $t^pe^{-\frac{\tau}{2}t^2}\leq (\frac{p}{\tau e})^{p/2}$ that holds for $\tau\in(0,1)$ and $t\geq 0$. The function $\tau(1-\tau)^{3+n/2}$ attains its maximum at $\tau^* = \frac{2}{8+n}$ and 
\[
{\tau^*}(1-\tau^*)^{3+n/2} = \frac{2}{(8+n)}\left(\frac{6+n}{8+n}\right)^{\frac{6+n}{2}}\geq \frac{2}{(8+n) e}.
\]
Customizing the above inequality for $\tau^*$ we get
\beq \label{eight2six:convert}
\EE\left[\langle u,\! \nabla^2_{xx} q~\! u\rangle^2(\theta^\top u)^2\|u\|^2 \right] \leq (8+n)\EE \left[\langle u,\! \nabla^2_{xx} q~\! u\rangle^2(\theta^\top u)^2\right].
\eeq 
Using the 6-th moment result of \cite{magnus1978moments} we know that for symmetric matrices $A$, $B$ and $C$:
\begin{align}\notag 
\EE u^\top A u u^\top B u u^\top C u = &\tr(A)\tr(B)\tr(C) + 2\tr(A)\tr(BC)\\& + 2\tr(B)\tr(AC) + 2\tr(C)\tr(AB) + 8\tr(ABC), \label{magnus:result}
\end{align}
and therefore
\begin{align}\notag 
\EE\left[ \langle u,\! \nabla^2_{xx} q~\! u\rangle^2(\theta^\top u)^2\right] &= \left[\tr(\nabla^2_{xx} q)\right]^2\|\theta\|^2 + 2\left\|\nabla^2_{xx} q\right\|^2_F\|\theta\|^2 + 4\tr(\nabla^2_{xx} q) \theta^\top \nabla^2_{xx} q ~\!\theta + 8  \left\| \nabla^2_{xx} q~\! \theta\right\|^2 \\ \notag  &\leq \left( n\|\nabla^2_{xx} q\|_F^2 + 2\|\nabla^2_{xx} q\|_F^2 + 4\sqrt{n}\|\nabla^2_{xx} q\|_F\|\nabla^2_{xx} q\| + 8\|\nabla^2_{xx} q\|^2\right)\|\theta\|^2\\ & \leq (5n+10)\|\nabla^2_{xx} q\|_F^2\|\theta\|^2,\label{ExpInner2Inner2}
\end{align}
where the second inequality uses the fact that for a symmetric matrix $A$, 
\begin{equation}\label{trace:ineq}
   \tr(A)\leq \sqrt{\rank(A)}\|A\|_F. 
\end{equation}
Together with \eqref{eight2six:convert} we conclude that 
\beq \label{thirdterm:bound}
\EE \left[\langle \nabla_{xx}^2 q~\! u,u\rangle^2(\theta^\top u)^2\|u\|^2\right]\leq 5(8+n)(n+2)\|\nabla_{xx}^2 q\|_F^2\|\theta\|^2.
\eeq 
For the second  term in \eqref{T2:expand} appealing once more to \eqref{cookbook:result:4thorder} we have
\begin{align}\notag 
\EE\left[ \langle u,\! \nabla^2_{xx} q~\! u\rangle^2\right] &= 2\left\|\nabla_{xx}^2 q\right\|_F^2 + \left[\tr\left(\nabla_{xx}^2 q\right)\right]^2\\&\leq (n+2)\left\|\nabla_{xx}^2 q\right\|_F^2.\label{expHessian:prod}
\end{align}
Using the results of \eqref{expHessian:prod} and \eqref{thirdterm:bound} in \eqref{T2:expand} yields 
\begin{align}
 \EE\left[ T_2(u) \right] \leq \frac{15}{2} (n+6)^2\left\|\nabla_{xx}^2 q\right\|_F^2\|\theta\|^2.\label{T2:Bound}
\end{align}
Next, focusing on $T_3(u,v)$ we have
\begin{equation}\label{T3:expand}
    \EE\left[T_3(u,v)\right] = \frac{3\mu^2}{\eta^2} \EE\left[\langle u,\! \nabla^2_{xy}q~\! v \rangle^2 (\theta^\top u)^2\|u\|^2\right] + \frac{3\mu^2}{\eta^2} \EE\left[ \langle u,\! \nabla^2_{xy}q~\! v \rangle^2 \right] \|\theta\|^2
\end{equation}
The first expectation on the right-hand side of \eqref{T3:expand} can be bounded as follows
\begin{align}\notag 
    \EE\left[\langle u,\! \nabla^2_{xy}q~\! v \rangle^2 (\theta^\top u)^2\|u\|^2\right] & = \EE_u\left[ \EE_v\left[u^\top \nabla^2_{xy}q~\! vv^\top \nabla^2_{xy}q^\top u (\theta^\top u)^2\|u\|^2\right] \right]\\ \notag 
 &=\EE_u\left[ u^\top \nabla^2_{xy}q~\! \nabla^2_{xy}q^\top u u^\top \theta\theta^\top u u^\top u \right]\\ \notag  &= (n+4)\|\nabla^2_{xy}q\|_F^2\|\theta\|^2  + 2(n+4)\theta^\top \nabla^2_{xy}q~\! \nabla^2_{xy}q^\top \theta \\&\leq 3(n+4)\|\nabla^2_{xy}q\|_F^2\|\theta\|^2 \label{T3:firstTerm}
\end{align}
where we used \eqref{magnus:result} for in the third equality. The second expectation on the right-hand side of \eqref{T3:expand} can be calculated as
\begin{align}\notag 
    \EE\left[\langle u,\! \nabla^2_{xy}q~\! v \rangle^2\right] & = \EE_u\left[ \EE_v\left[u^\top \nabla^2_{xy}q~\! vv^\top \nabla^2_{xy}q^\top u \right] \right]\\ \notag 
 &=\EE_u\left[ u^\top \nabla^2_{xy}q~\! \nabla^2_{xy}q^\top u  \right]\\ \notag  &= \|\nabla^2_{xy}q\|_F^2, 
\end{align}
which together with \eqref{T3:firstTerm} and \eqref{T3:expand} implies that 
\begin{equation}\label{T3:bound}
    \EE\left[T_3(u,v)\right] \leq \frac{3\mu^2}{\eta^2}(3n+13)\|\nabla^2_{xy}q\|_F^2\|\theta\|^2. 
\end{equation}
Finally, 
\begin{align}\notag
    \EE\left[T_4(u,v)\right] &=\frac{3\mu^4}{2\eta^4}\EE_v\left[\langle v,\! \nabla^2_{yy}q ~\! v \rangle^2\right]\EE_u\left[(\theta^\top u)^2\|u\|^2+\|\theta\|^2\right] \\& \leq \frac{3\mu^4}{2\eta^4}(m+2)(n+3)\|\nabla^2_{yy}q\|_F^2\|\theta\|^2,\label{T4:bound}
\end{align}
Applying an expectation to \eqref{T1-T4Expansion}, and using the bounds \eqref{T4:bound}, \eqref{T3:bound}, \eqref{T2:Bound}, and \eqref{T1:Bound} immediately implies the claim in Proposition \ref{hessMomentBounds}.

We now show part b). By definition \eqref{stochHessxy},
\begin{align}\notag 
    \left\|  \tilde \nabla_{xy}^2 q_{\eta,\mu}(x,y) \theta  \right\|^2 = \frac{\left(q(x+\eta u, y+\mu v)\! +\! q(x-\eta u,y-\mu v)\!-\! 2q(x,y )\right)^2}{\eta^2\mu^2}\|u\|^2(\theta^\top v)^2.
\end{align}
In the first part of Proposition \ref{hessMomentBounds} proof we showed that
\begin{align*}
    ((q(x+\eta u, y+\mu v)\!  + \! q(x-\eta u,y-\mu v)\!-\! 2q(x,y ))^2  \!\leq\!&~ 8 L_{2,q}^2\left(\|\eta u\|^6 + \|\mu v\|^6\right)+ 6\eta^4\langle u,\! \nabla^2_{xx} q(x,y) u\rangle^2 \\& +\! 12\eta^2 \mu^2 \langle u,\! \nabla^2_{xy}q(x,y) v \rangle^2 \!+\! 6\mu^4 \langle v,\! \nabla^2_{yy}q(x,y) v \rangle^2.
\end{align*}
Using this result and dropping the dependency of $q$ to $(x,y)$ for brevity, we get 
\begin{align}\notag 
    \left\|  \tilde \nabla_{xy}^2 q_{\eta,\mu} \theta  \right\|^2 \leq & \underbrace{8L_{2,q}^2\left( \frac{\eta^4}{\mu^2}\|u\|^6 + \frac{\mu^4}{\eta^2}\|v\|^6\right)\|u\|^2(\theta^\top v)^2}_\text{$T_1(u,v)$} + \underbrace{\frac{6\eta^2}{\mu^2}  \langle u,\! \nabla^2_{xx} q~\! u\rangle^2 \|u\|^2(\theta^\top v)^2}_\text{$T_2(u,v)$} \\ & + \underbrace{12 \langle u,\! \nabla^2_{xy}q~\! v \rangle^2 \|u\|^2(\theta^\top v)^2}_\text{$T_3(u,v)$}  + \underbrace{\frac{6\mu^2}{\eta^2} \langle v,\! \nabla^2_{yy}q ~\! v \rangle^2 \|u\|^2(\theta^\top v)^2}_\text{$T_4(u,v)$}. \label{2ndmomentExpandxy}
\end{align}
In the sequel we bound the expectation of $T_1(u,v)$ through $T_4(u,v)$. For the first term we have
\begin{align}\notag 
    \EE\left[T_1(u,v)\right] &= 8L_{2,q}^2\EE\left[ \frac{\eta^4}{\mu^2}\|u\|^8(\theta^\top v)^2 + \frac{\mu^4}{\eta^2}\|u\|^2\|v\|^6(\theta^\top v)^2 \right]\\&\leq 8L_{2,q}^2\left[ \frac{\eta^4}{\mu^2}(n+8)^4 + \frac{2\mu^4}{\eta^2} n(m+12)^3\right]\|\theta\|^2,\label{T1Bound-xy}
\end{align}
where we used Lemma \ref{lemmaExpectationNormPower} and \eqref{k-momentTimesNorm}. 

For $T_2(u,v)$ we have
\begin{align}\notag 
    \EE\left[ T_2(u,v) \right] &= \frac{6\eta^2}{\mu^2} \EE_u\left[\langle u,\! \nabla^2_{xx} q~\! u\rangle^2\|u\|^2 \right] \EE_v\left[ (\theta^\top v)^2\right]\\ \notag &= \frac{6\eta^2}{\mu^2} \EE_u\left[\langle u,\! \nabla^2_{xx} q~\! u\rangle^2\|u\|^2 \right] \|\theta\|^2\\ \notag &=
    \frac{6\eta^2}{\mu^2}\left((n+4)\left[\tr(\nabla^2_{xx} q)\right]^2 + (2n+8)\left\|\nabla^2_{xx} q\right\|^2_F\right)\|\theta\|^2 \\ & \leq \frac{6\eta^2}{\mu^2}(n+4)(n+2)\|\nabla^2_{xx} q\|_F^2\|\theta\|^2,  
    \label{T2:expandxy}
\end{align}
where we used \eqref{magnus:result} for the third equality and \eqref{trace:ineq} for the last inequality.

The expectation of $T_3(u,v)$ can be rewritten as
\begin{align}\notag 
    \EE\left[ T_3(u,v)\right] &= 12\EE\left[v^\top \nabla^2_{xy}q^\top u u^\top u u^\top \nabla^2_{xy}q v v^\top \theta\theta^\top v \right]\\\notag &=12\EE_v\left[ v^\top \nabla^2_{xy}q^\top \EE_u\left[u u^\top u u^\top \right]\nabla^2_{xy}q v v^\top \theta\theta^\top v\right]\\ \notag & = 12(n+2)\EE\left[ v^\top \nabla^2_{xy}q^\top \nabla^2_{xy}q v v^\top \theta\theta^\top v\right]\\\notag & = 12(n+2)\left(2\theta^\top \nabla^2_{xy}q^\top \nabla^2_{xy}q\theta +  \left\|\nabla^2_{xy}q\right\|_F^2\|\theta\|^2\right)\\&\leq 36(n+2)  \left\|\nabla^2_{xy}q\right\|_F^2\|\theta\|^2,\label{T3:boundxy}
\end{align}
where we used \eqref{cookbook:result:4thorder} in the fourth equality. 

Finally, the last term in \eqref{2ndmomentExpandxy} can be bounded as
\begin{align}\notag 
    \EE\left[T_4(u,v)\right] &= \frac{6\mu^2}{\eta^2}\EE_u\left[\|u\|^2\right]\EE_v\left[\langle v,\! \nabla^2_{yy}q ~\! v \rangle^2 (\theta^\top v)^2\right]\\&\leq \frac{30\mu^2}{\eta^2}n(m+2)\|\nabla^2_{yy} q\|_F^2\|\theta\|^2,\label{T4:Boundxy}
\end{align}
where the expectation with respect to $v$ is bounded in an identical way as \eqref{ExpInner2Inner2}. 

Applying the individual term bounds in \eqref{T4:Boundxy}, \eqref{T3:boundxy}, \eqref{T2:expandxy} and \eqref{T1Bound-xy} to \eqref{2ndmomentExpandxy} validates the advertised result.

\subsection{Proof of Proposition 
\ref{zeroth-order bilevel}}\label{zeroth-order bilevel:proof}
We consider proving the more comprehensive version of Proposition \ref{zeroth-order bilevel} as follows.
\begin{proposition}\label{zeroth-order bilevelFULL}
Consider the bilevel optimization problem \eqnok{main_prob_st} and its Gaussian smooth approximation \eqref{main_prob_zst}. 
\begin{itemize}
    \item [a)] If $f\in\mathcal{C}^0(X\times \R^m,L_{0,f})$ and $g\in\mathcal{C}^0(X\times \R^m,L_{0,g})$, then
    \begin{align}
    &\|y_{\eta_2,\mu_2}^*(x)-y^*(x)\|^2 \le \frac{4L_{0,g}}{\lambda_g} \left(\eta_2^2n + \mu_2^2 m\right)^{\frac{1}{2}}, \label{distysL0FULL}\\
      &|\psi_{\eta,\mu}(x) - \psi(x)| \le L_{0,f} \left[2\sqrt{\frac{L_{0,g}}{\lambda_g}} \left(\eta_2^2n + \mu_2^2 m\right)^{\frac{1}{4}} +\left(\eta_1^2n + \mu_1^2 m\right)^{\frac{1}{2}} \right]. \label{diffPsi1FULL}
    \end{align}
    
    \item [b)] If $f\in\mathcal{C}^1(X\times \R^m,L_{1,f})$ and $g\in\mathcal{C}^1(X\times \R^m,L_{1,g})$, then
    \begin{align}
   & \|y_{\eta_2,\mu_2}^*(x)-y^*(x)\|^2 \leq \frac{2 L_{1,g}}{\lambda_g }\left(\eta_2^2n + \mu_2^2 m\right), \label{distysL1FULL}\\ 
    &\|\nabla \psi_{\eta,\mu}(x) - \nabla \psi(x)\| \le L_{1, f} \sqrt{\frac{2 L_{1,g}}{\lambda_g }\left(\eta_2^2n + \mu_2^2 m\right)}\nonumber \\
    & \qquad \qquad \qquad \qquad \qquad + \frac{L_{1,f}}{2}\left(\eta_1(n+3)^{\frac{3}{2}} + \frac{\mu_1^2}{\eta_1} m n^{\frac{1}{2}} + \frac{\eta_1^2}{\mu_1} n m^{\frac{1}{2}}+ \mu_1(m+3)^{\frac{3}{2}} \right).\label{grad_diffFULL}
    \end{align}
    
\end{itemize}
\end{proposition}
First, note that Assumption \ref{fg_assumption}(b) implies that for a fixed $x$, any pair of points $y_1, y_2\in\R^m$ obey:
\begin{equation}\label{strConvg}
    g(x,y_2) - g(x,y_1)\geq \nabla_y g(x,y_1)^\top (y_2-y_1) + \frac{\lambda_g}{2}\|y_2-y_1\|^2.
\end{equation}
By part (a) of Proposition \ref{general_prop}, $g_{\eta_2,\mu_2}(x,y_2)$ also has a similar modulus of convexity, and
\begin{equation}\label{strConvgem}
    g_{\eta_2,\mu_2}(x,y_2) - g_{\eta_2,\mu_2}(x,y_1)\geq \nabla_y g_{\eta_2,\mu_2}(x,y_1)^\top (y_2-y_1) + \frac{\lambda_g}{2}\|y_2-y_1\|^2.
\end{equation}
Since $y^*(x)$ and $y_{\eta_2,\mu_2}^*(x)$ are respectively the optimal points for $g(x,y)$ and $g_{\eta_2,\mu_2}(x,y)$, appealing to \eqref{strConvg} and \eqref{strConvgem}, one gets
\begin{align*}
    \frac{\lambda_g}{2}\|y_{\eta_2,\mu_2}^*(x)-y^*(x)\|^2 &\leq g(x,y_{\eta_2,\mu_2}^*(x)) - g(x,y^*(x)),\\
    \frac{\lambda_g}{2}\|y_{\eta_2,\mu_2}^*(x)-y^*(x)\|^2 &\leq g_{\eta_2,\mu_2}(x,y^*(x)) - g_{\eta_2,\mu_2}(x,y_{\eta_2,\mu_2}^*(x)).
\end{align*}
Adding the two inequalities, we are able to bound the distance between $y_{\eta_2,\mu_2}^*(x)$ and $y^*(x)$ as
\begin{align}\notag 
    \|y_{\eta_2,\mu_2}^*(x)-y^*(x)\|^2 &\leq \frac{2}{\lambda_g}\left |g_{\eta_2,\mu_2}(x,y_{\eta_2,\mu_2}^*(x)) - g(x,y_{\eta_2,\mu_2}^*(x))\right | + \frac{2}{\lambda_g}\left |g_{\eta_2,\mu_2}(x,y^*(x)) - g(x,y^*(x))\right |\\ &\leq \frac{4L_{0,g}}{\lambda_g} \left(\eta_2^2n + \mu_2^2 m\right)^{\frac{1}{2}},\label{distysL0} 
\end{align}
where the second inequality is thanks to part (a) of Proposition \ref{propNestApprox}. Alternatively, an application of part (b) of the same proposition implies \eqref{distysL1FULL}.
In view of this and part (b) of Proposition \ref{general_prop} we are able to bound the discrepancy  between $\psi_{\eta,\mu}(x)$  and  $\psi(x)$ as 
\begin{align}\notag 
|\psi_{\eta,\mu}(x) - \psi(x)|&=
    \left|f_{\eta_1,\mu_1}(x,y_{\eta_2,\mu_2}^*(x)) - f(x,y^*(x))\right| \\ \notag 
    &\le \left|f_{\eta_1,\mu_1}(x,y_{\eta_2,\mu_2}^*(x)) - f_{\eta_1,\mu_1}(x,y^*(x))\right|
    +\left|f_{\eta_1,\mu_1}(x,y^*(x)) - f(x,y^*(x))\right|\\ \notag  &\leq L_{0,f}   \left\|y_{\eta_2,\mu_2}^*(x)-y^*(x)\right\| + L_{0,f}\left(\eta_1^2n + \mu_1^2 m\right)^{\frac{1}{2}}\\&\leq 2L_{0,f}\sqrt{\frac{ L_{0,g}}{\lambda_g}} \left(\eta_2^2n + \mu_2^2 m\right)^{\frac{1}{4}} + L_{0,f}\left(\eta_1^2n + \mu_1^2 m\right)^{\frac{1}{2}},\notag
\end{align}
which supports the claim in \eqref{diffPsi1FULL}. 

Finally, under the assumptions of part (b), and using part (b) of Proposition \ref{propNestApprox} we obtain
\begin{align*}
    \big \|\nabla f_{\eta,\mu}(x,&y_{\eta_2,\mu_2}^*(x)) - \nabla f(x,y^*(x))\big\| \\
    &\le \left \|\nabla f_{\eta,\mu}(x,y_{\eta_2,\mu_2}^*(x)) - \nabla f_{\eta,\mu}(x,y^*(x))\right\| +\left\|\nabla f_{\eta,\mu}(x,y^*(x)) - \nabla f(x,y^*(x))\right\|\\
    &\le L_{1, f} \left\|y_{\eta_2,\mu_2}^*(x) - y^*(x) \right\| + \frac{L_{1,f}}{2}\left(\eta_1(n+3)^{\frac{3}{2}} + \frac{\mu_1^2}{\eta_1} m n^{\frac{1}{2}} + \frac{\eta_1^2}{\mu_1} n m^{\frac{1}{2}}+ \mu_1(m+3)^{\frac{3}{2}} \right),
\end{align*}
which together with \eqref{distysL1FULL}, imply \eqref{grad_diffFULL}.

\subsection{Proof of Theorem \ref{sgd:hessInv}}\label{sgd:hessInv:proof}
The proof is comprised of two parts. In the first part (Lemma \ref{sgd:quadratic}), we consider the quadratic minimization problem
\begin{equation}
    \minimize_z \frac{1}{2}z^\top H z - h^\top z, \label{quad:prog}
\end{equation}
where $H = \EE_u [H(u)]\succ 0$,  $h = \EE_w [h(w)]$, and $u$ and $w$ are some random vectors. We provide convergence results for this problem considering the SGD updates
\begin{equation}
  z_{k+1} = z_k - \gamma_k (H(u_k)z_k - h(w_k)).\label{sgd:iterate}  
\end{equation}
As will be stated in Lemma \ref{sgd:quadratic} below, the convergence results are in terms of some moment terms $\bar V$ and $V_H$. In the second part of proof, we focus on the customized problem \eqref{zbar:program}, and show that as an instance of \eqref{quad:prog}, the conditions stated by Lemma \ref{sgd:quadratic} can be established for parameters that will be derived.

\begin{lemma}\label{sgd:quadratic}
Consider the minimization \eqref{quad:prog} and the SGD updates \eqref{sgd:iterate}, such that $H = \EE_u [H(u)]$ and  $h = \EE_w [h(w)]$. Assume that
$0\prec L_{\min}I\preceq H\preceq L_{\max}I$,
and for a given $z$ and the optimal solution $z^* = H^{-1}h$:
        \begin{equation}
            \EE_u \|H(u)z\|^2\leq V_H\|z\|^2, ~~~~~~\EE_{u,w} \|H(u)z^* - h(w)\|^2\leq \bar V,
        \end{equation}  
where $V_H$ and $\bar V$ are some positive scalars. Choose $0<\gamma < \min\left\{ \frac{2}{L_{\min}+L_{\max}}, \frac{L_{\min}}{2V_H}\right\}$ and set $\gamma_k = \gamma$, then
\begin{align}
    \label{eqSGD:Var}
    \EE \| z_N - z^*\|^2 &\leq (1-\gamma L_{\min})^N\|z_0 - z^*\|^2 + \frac{2\gamma \bar V}{L_{\min}},\\ 
    \|\EE z_N - z^*\|^2 &\leq (1-\gamma L_{\min})^N\|z_0 - z^*\|^2, \label{eqSGD:bias}
\end{align}
where $z_N$ is the $N$-th SGD update. 
\end{lemma}
\begin{proof}
We start by proving \eqref{eqSGD:Var}. Let $\EE_{|k}$ denote the conditional expectation given $z_k$. Subtracting $z^*$ from both sides of \eqref{sgd:iterate} and taking a squared norm gives
\begin{align*}
    \EE_{|k}\|z_{k+1}-z^*\|^2 = \|z_k-z^*\|^2 -2\gamma(z_k-z^*)^\top H(z_k-z^*) + \gamma^2 \EE_{|k} \|H(u_k)z_k - h(w_k)\|^2. 
\end{align*}
The last squared term above can be bounded as
\begin{align*}
    \EE_{|k} \|H(u_k)z_k - h(w_k)\|^2 &= \EE_{|k}\|H(u_k)(z_k-z^*) + H(u_k)z^* - h(w_k)\|^2\\&\leq 2\EE_{|k}\|H(u_k)(z_k-z^*)\|^2 +2\EE_{|k} \|H(u_k)z^* - h(w_k)\|^2,\\&\leq 2V_H\|(z_k-z^*)\|^2 + 2\bar V,
\end{align*}
and therefore 
\begin{align*}
    \EE_{|k}\|z_{k+1}-z^*\|^2 &\leq  \|z_k-z^*\|^2 -2\gamma (L_{\min} - \gamma V_H )\|z_k-z^*\|^2 + 2\gamma^2 \bar V\\&  \leq (1-\gamma L_{\min})\|z_k-z^*\|^2 + 2\gamma^2 \bar V
\end{align*}
where the second inequality is thanks to $\gamma< L_{\min}/(2V_H)$ which implies $2\gamma (L_{\min} - \gamma V_H )> \gamma L_{\min}$. Applying an expectation with respect to $z_k$ to both sides gives
\begin{equation}\label{telescop:eq}
    \EE\| z_{k+1}-z^*\|^2\leq (1-\gamma L_{\min})\EE\| z_k-z^*\|^2 + 2\gamma^2\bar V.
\end{equation}
Notice that thanks to $\gamma\leq 2/(L_{\min}+L_{\max})\leq 1/L_{\min}$, we have $0<(1-\gamma L_{\min})<1$, and recursively applying \eqref{telescop:eq} gives
\begin{align*}
    \EE\|z_N - z^*\|^2 &\leq (1-\gamma L_{\min})^N\|z_0 - z^*\|^2 + 2\gamma^2 \bar V\sum_{k=0}^{N-1} (1-\gamma L_{\min})^k \\& \leq  (1-\gamma L_{\min})^N\|z_0 - z^*\|^2 + \frac{2\gamma \bar V}{L_{\min}},
\end{align*}
which completes the proof of \eqref{eqSGD:Var}.

To prove \eqref{eqSGD:bias}, the SDG update implies that
\[
\EE_{|k} z_{k+1} - z^* = z_k -z^* -  \gamma(Hz_k - h) = z_k -z^* -  \gamma H(z_k - z^*).
\]
Taking an expectation of both side with respect to $z_k$ and taking a norm gives
\begin{equation}
    \|\EE z_{k+1} - z^*\| = \|(I -\gamma H)\EE(z_k-z^*)\| \leq  \|I -\gamma H\|\|\EE z_k-z^*\|.
\end{equation}
Since $\gamma\leq 2/(L_{\min}+L_{\max})$, by the Weyl's inequality, all the eigenvalues of $I - \gamma H$ are within the interval $[0, 1-\gamma L_{\min}]$ and we have $ \|\EE z_{k+1} - z^*\| \leq (1-\gamma L_{\min})\|\EE z_k-z^*\|$, or
\begin{equation}\label{eqWeyl}
\|\EE z_N - z^*\|^2 \leq (1-\gamma L_{\min})^N\|z_0 - z^*\|^2,
\end{equation}
which completes the lemma proof. 
\end{proof}
In the remainder of Theorem \ref{sgd:hessInv} proof, we customize Lemma \ref{sgd:quadratic} for the quadratic problem \ref{zbar:program}. 

To establish the conditions $H = \EE_u [H(u)]$ and  $h = \EE_w [h(w)]$ of Lemma \ref{sgd:quadratic} for problem \ref{zbar:program}, note that
$\tilde \nabla J(z)$ can be written as $ \tilde \nabla_{yy}^2 G_{\eta_2,\mu_2}z - \tilde \nabla F_{\eta_1,\mu_1}$ where
\begin{align}\notag 
\tilde \nabla_{yy}^2 G_{\eta_2,\mu_2} &:=  \Big( v v^\top - I_m\Big) \left[\frac{G(\bar x+\eta_2 u, \bar y+\mu_2 v, \zeta) + G(\bar x-\eta_2 u,\bar y-\mu_2 v, \zeta)-2G(\bar x,\bar y,\zeta)}{2\mu_2^2} \right],\\   \tilde \nabla_y F_{\eta_1,\mu_1}&:=v'\!\left[  \frac{F(\bar x+\eta_1 u', \bar y+\mu_1 v', \xi) - F(\bar x, \bar y,\xi)}{\mu_1} \ \right].
\end{align}
By the definition we have
\begin{align}\notag 
\mathbb{E}_\zeta \tilde \nabla_{yy}^2 G_{\eta_2,\mu_2} =  \Big( v v^\top - I_m\Big) \left[\frac{g(\bar x+\eta_2 u, \bar y+\mu_2 v) + g(\bar x+\eta_2 u,\bar y-\mu_2 v)-2g(\bar x+\eta_2 u,\bar y)}{2\mu_2^2} \right],
\end{align}
and using \eqref{Hessian_y} we get
\[
\mathbb{E}_{u,v}\mathbb{E}_\zeta \tilde \nabla_{yy}^2 G_{\eta_2,\mu_2}= \nabla^2_{yy} g_{\eta_2,\mu_2}(\bar x,\bar y).
\]
In a similar fashion we can show that $\mathbb{E}_{u',v'}\mathbb{E}_\xi \tilde \nabla_y F_{\eta_1,\mu_1}= \nabla_{y} f_{\eta_1,\mu_1}(\bar x,\bar y)
$. 

Next, to find the quantities $L_{\min}$ and $L_{\max}$ of Lemma \ref{sgd:quadratic} that correspond to problem \ref{zbar:program}, notice that for a test vector $y_0$:
\[
y_0^\top \nabla^2_{yy} g_{\eta_2,\mu_2}(\bar x,\bar y)  y_0 = (0,y_0)^\top \nabla^2 g_{\eta_2,\mu_2}(\bar x,\bar y)(0,y_0) \geq \lambda_g \|y_0\|^2,
\]
where the inequality is thanks to part (a) of Proposition \ref{general_prop}. In a similar fashion, this time using part (b) of Proposition \ref{general_prop} we get $y_0^\top \nabla^2_{yy} g_{\eta_2,\mu_2}(\bar x,\bar y)  y_0\leq L_{2,g}\|y_0\|^2$. Together, they imply that
\begin{equation}
  \lambda_gI  \preceq \nabla^2 g_{\eta_2,\mu_2}(\bar x,\bar y) \preceq L_{2,g}I,
\end{equation}
indicating that when applied to problem \ref{zbar:program}, $L_{\min}$ and $L_{\max}$ of Lemma \ref{sgd:quadratic} need to be set to $\lambda_g$ and $L_{2,g}$, respectively. 

As the last two steps, for $z$ being a given vector and $\bar z$ the optimal solution in \eqref{zbar:program}, we bound  $\EE\| \tilde \nabla_{yy}^2 G_{\eta_2,\mu_2}z\|^2$ and $\EE\|  \tilde \nabla_{yy}^2 G_{\eta_2,\mu_2}\bar z - \tilde \nabla_y F_{\eta_1,\mu_1}\|^2$ to find $V_H$ and $\bar V$.   

Using Proposition \ref{hessMomentBounds} customized for $\tilde \nabla_{yy}^2$, we have
\begin{align*}\notag 
\EE_{u,v}\Big\|  \tilde \nabla_{yy}^2 G_{\eta_2,\mu_2}z  \Big\|^2  \leq & ~2L_{2,G}^2\left( 2\mu_2^2(m+16)^4 + \frac{\eta_2^6}{\mu_2^4}(n+6)^3 (m+3)\right)\|z\|^2\\ \notag &+ \bigg( \frac{15}{2} (m+6)^2\left\|\nabla_{yy}^2 G(\bar x,\bar y,\zeta)\right\|_F^2 + \frac{3\eta_2^2}{\mu_2^2}(3m+13)\|\nabla^2_{yx}G(\bar x,\bar y,\zeta)\|_F^2 \\& \qquad\qquad\qquad\quad  \qquad\qquad\qquad+ \frac{3\eta_2^4}{2\mu_2^4}(n+2)(m+3)\|\nabla^2_{xx}G(\bar x,\bar y,\zeta)\|_F^2\bigg)\|z\|^2.
\end{align*}
Assuming $\eta_2/\mu_2=\mathcal{O}(1)$ 
\begin{align*}\notag 
\EE_{u,v}\Big\|  \tilde \nabla_{yy}^2 G_{\eta_2,\mu_2}z  \Big\|^2 \!\!\leq & \! ~ c\mu_2^2 L_{2,G}^2m\left( m+n\right)^3\|z\|^2 \!+\!c m\left(m+n\right)\!\Big(\!\left\|\nabla_{yy}^2 G(\bar x,\bar y,\zeta)\right\|_F^2\!+ \!2\!\left\|\nabla_{yx}^2 G(\bar x,\bar y,\zeta)\right\|_F^2  \\ \notag &~~~~~~~~~~~~~~~~~~~~~~~~~~~~~~~~~~~~~~~~~~~~~~~~~~~~~+\left\|\nabla_{xx}^2 G(\bar x,\bar y,\zeta)\right\|_F^2\Big)\|z\|^2 \\ \notag =&~ c\mu_2^2 L_{2,G}^2m\left( m+n\right)^3\|z\|^2 + c m\left(m+n\right)\left\|\nabla^2 G(\bar x,\bar y,\zeta)\right\|_F^2\|z\|^2,
\end{align*}
where $c$ is a sufficiently large constant. Finally, taking an expectation of both sides of the inequality above with respect to $\zeta$, and using the fact that for a given matrix $A$, $\|A\|_F\leq \sqrt{\rank(A)}\|A\|$ yields 
\begin{align}\notag 
\EE\Big\|  \tilde \nabla_{yy}^2 G_{\eta_2,\mu_2}z  \Big\|^2 &= \EE_{\zeta}\EE_{u,v} \Big\|  \tilde \nabla_{yy}^2 G_{\eta_2,\mu_2}z  \Big\|^2 \\ \notag &\leq  c\mu_2^2 L_{2,G}^2m\left( m+n\right)^3\|z\|^2 + c m\left(m+n\right)\EE_\zeta \left[\left\|\nabla^2 G(\bar x,\bar y,\zeta)\right\|_F^2\right]\|z\|^2,
\\ \notag \notag &=  c\mu_2^2 L_{2,G}^2m\left( m+n\right)^3\|z\|^2 + c m\left(m+n\right)\left(\sigma_{2,G}^2 + \left\|\nabla^2 g(\bar x,\bar y)\right\|_F^2\right)\|z\|^2\\ 
 &\leq   c\mu_2^2 L_{2,G}^2m\left( m+n\right)^3\|z\|^2 + c m\left(m+n\right)^2\left(\frac{\sigma_{2,G}^2}{m+n} + L_{2,G}^2\right)\|z\|^2,\label{vHbound}
\end{align}
which is consistent with the $V_H$ expression in \eqref{def_VH}.

Finally, to find $\bar V$ we bound $\EE\|  \tilde \nabla_{yy}^2 G_{\eta_2,\mu_2}\bar z - \tilde \nabla F_{\eta_1,\mu_1}\|^2$ as follows
\begin{align}\notag 
    \EE\left\|  \tilde \nabla_{yy}^2 G_{\eta_2,\mu_2}\bar z - \tilde \nabla_y F_{\eta_1,\mu_1}\right \|^2 &= \mathbb{V}\left[ \tilde \nabla_{yy}^2 G_{\eta_2,\mu_2}\bar z - \tilde \nabla_y F_{\eta_1,\mu_1}\right]\\ \notag  &= \mathbb{V}\left[ \tilde \nabla_{yy}^2 G_{\eta_2,\mu_2}\bar z \right] + \mathbb{V}\left[  \tilde \nabla_y F_{\eta_1,\mu_1}\right] \\&\leq \EE\left\| \tilde \nabla_{yy}^2 G_{\eta_2,\mu_2}\bar z \right\|^2 + \EE\left\|  \tilde \nabla_y F_{\eta_1,\mu_1}\right\|^2 \label{vbar:expansion}
\end{align}
where the first equality is thanks to 
\[\EE\left[ \tilde \nabla_{yy}^2 G_{\eta_2,\mu_2}\bar z - \tilde \nabla_y F_{\eta_1,\mu_1}\right] = \nabla_{yy}^2 g_{\eta_2,\mu_2}(\bar x, \bar y)\bar z-  \nabla_y f_{\eta_1,\mu_1}(\bar x, \bar y) = 0,
\]
and the second equality is thanks to the independence of $\tilde \nabla_{yy}^2 G_{\eta_2,\mu_2}\bar z$ and $\tilde \nabla_y F_{\eta_1,\mu_1}$. The magnitude of $\bar z$ can be bounded as
\beq\notag 
\|\bar z\| = \left\|\left[\nabla_{yy}^2 g_{\eta_2,\mu_2}(\bar x, \bar y)\right]^{-1}  \nabla_y f_{\eta_1,\mu_1}(\bar x, \bar y)\right\|\leq \left\|\left[\nabla_{yy}^2 g_{\eta_2,\mu_2}(\bar x, \bar y)\right]^{-1} \right\| \left\| \nabla_y f_{\eta_1,\mu_1}(\bar x, \bar y)\right \|\leq \frac{L_{0,f}}{\lambda_g},
\eeq 
which combined with \eqref{vHbound} offers a bound for  the first term in \eqref{vbar:expansion} as
\begin{align}
\EE\Big\|  \tilde \nabla_{yy}^2 G_{\eta_2,\mu_2}\bar z  \Big\|^2 
 \leq   c \frac{\mu_2^2 L_{0,f}^2L_{2,G}^2}{\lambda_g^2}m\left( m+n\right)^3 + c\frac{L_{0,f}^2}{\lambda_g^2} m\left(m+n\right)^2\left(\frac{\sigma_{2,G}^2}{m+n} + L_{2,G}^2\right).\label{vbarfirstterm}
\end{align}

For the second term in \eqref{vbar:expansion}, setting $\eta_1=0$ in Corollary \ref{corMomentBounds} we obtain
\begin{align*}
\EE_{u',v'}  \left \|\tilde \nabla_y F_{0,\mu_1} \right \|^2  &\leq \mu_1^2 L_{1,F}^2(m+6)^3+  4(m+2) \|\nabla_y  F(\bar x, \bar y,\xi)\|^2\\ &\leq c'\mu_1^2 L_{1,F}^2m^3+  c'm \|\nabla_y  F(\bar x, \bar y,\xi)\|^2,
\end{align*}
where $c'$ is a sufficiently large constant. Taking an expectation of both sides of the inequality with respect to $\xi$ yields 
\begin{align}\notag 
\EE \left \|\tilde \nabla_y F_{\eta_1,\mu_1} \right \|^2  &= \EE_{\xi}\EE_{u',v'} \left \|\tilde \nabla_y F_{\eta_1,\mu_1} \right \|^2 \\ \notag &\leq  c'\mu_1^2 L_{1,F}^2 m^3 + c'm \EE_\xi\left[ \|\nabla  F(\bar x, \bar y,\xi)\|^2\right],
\\ \notag &=  c'\mu_1^2 L_{1,F}^2 m^3 + c'm \left(\sigma_{1,F}^2 + \left\|\nabla f(\bar x,\bar y)\right\|^2\right)
\\ \notag &\leq   c'\mu_1^2 L_{1,F}^2 m^3 + c'm \left(\sigma_{1,F}^2 + L_{0,f}^2\right).
\end{align}
Using this inequality along with \eqref{vbarfirstterm} in \eqref{vbar:expansion} verifies the advertised expression of $\bar V$ in \eqref{def_Vbar}. 


\end{document}